\documentclass{CSML}

\def\dOi{12(4:4)2016}
\lmcsheading%
{\dOi}
{1--35}
{}
{}
{Mar.~19, 2016}
{Oct.~18, 2016}
{}

\ACMCCS{{\bf [Theory of computation]}: Logic; {\bf [Mathematics of
    computing]}: Continuous mathematics---Topology.} 

\usepackage{hyperref}
\usepackage{amssymb}

\theoremstyle{plain}

\theoremstyle{plain}
\newtheorem{theorem}[thm]{Theorem}

\newtheorem{lemma}[thm]{Lemma}

\theoremstyle{definition}
\newtheorem{definition}[thm]{Definition}
\newtheorem{remark}[thm]{Remark}
\newtheorem{example}[thm]{Example}

\newtheorem*{claim}{Claim}

\newcommand{\fr}{\mbox{}^\smallfrown}
\newcommand{\pair}[1]{\langle #1 \rangle}

\newcommand{\upr}{\!\upharpoonright\!}

\newcommand{\om}{\omega}

\begin{document}

\title[Borel-Piecewise Continuous Reducibility]{Borel-Piecewise Continuous Reducibility for Uniformization Problems}

\author[T.~Kihara]{Takayuki Kihara}	
\address{Department of Mathematics, University of California, Berkeley, United States}	
\email{kihara@math.berkeley.edu}  
\thanks{The author was partially supported by a Grant-in-Aid for JSPS fellows.}






\begin{abstract}
  \noindent 
We study a fine hierarchy of Borel-piecewise continuous functions, especially, between closed-piecewise continuity and $G_\delta$-piecewise continuity.
Our aim is to understand how a priority argument in computability theory is connected to the notion of $G_\delta$-piecewise continuity, and then we utilize this connection to obtain separation results on subclasses of $G_\delta$-piecewise continuous reductions for uniformization problems on set-valued functions with compact graphs.
This method is also applicable for separating various non-constructive principles in the Weihrauch lattice.
\end{abstract}

\maketitle

\section{Introduction}

\subsection{Historical Background}\label{section:history}

For topological spaces $\mathcal{X}$ and $\mathcal{Y}$, a function $f:\mathcal{X}\to \mathcal{Y}$ is {\em $\sigma$-continuous} (or {\em countably continuous}) if there is a countable cover $\{\mathcal{X}_n\}_{n\in\omega}$ of $\mathcal{X}$ such that $f\upr \mathcal{X}_n$ is continuous for every $n\in\omega$.
If each $\mathcal{X}_n$ can be chosen as a $\mathbf{\Gamma}$ set, then $f$ is said to be {\em $\mathbf{\Gamma}$-piecewise continuous}.
It is clear that every $\sigma$-continuous Borel function is always Borel-piecewise continuous.
The notion of $\sigma$-continuity was first proposed by Luzin, who asked, in the early 20th century, whether every Borel function is $\sigma$-continuous.
Although Luzin's problem has been solved negatively, in recent years, the notion of $\sigma$-continuity itself has received increasing attention in descriptive set theory and related areas.
In these areas, researchers have accomplished an enormous amount of work connecting finite-level Borel functions and Borel-piecewise continuous functions (see \cite{MilCar15a,MilCar15b,Debs14,GreKihNg,JR82,KaMoSe,Kih15,MRos13,PawSab12,SemPhD,Solecki98}).
These works have also led us to the discovery that the notion of piecewise continuity plays a crucial role in the study of the hierarchy of Borel isomorphisms (see \cite{JayRog79a,KihPau}).

The hierarchies of closed-piecewise continuous functions have been extensively studied in various areas of mathematics and computer science, e.g., in the context of the levels of discontinuity \cite{deBre14,Hem08,Hert96,Pau10}, the subhierarchy of Baire-one-star functions \cite{Malek06,OMa77,Paw00}, and the mind-change hierarchy \cite{FreSmi93}.
The transfinite hierarchy of levels of discontinuity (numbers of mind-changes, etc.)~is actually useful for analyzing the Baire hierarchy of Borel functions. For instance, Solecki \cite[Theorem 3.1]{Solecki98} used a transfinite derivation process to obtain his dichotomy theorem for Baire-one functions, and Semmes \cite[Lemma 4.3.3]{SemPhD} introduced a higher level analog of a transfinite derivation process to prove his $G_\delta$-decomposition theorem for the $\Lambda_{2,3}$ functions (a subclass of the Baire-two functions).

The class of $\sigma$-continuous functions which are not closed-piecewise continuous is also found to have a crucial role in various fields. For instance, such a notion is closely associated with the notion of {\em countable-dimensionality} in infinite dimensional topology (see \cite{vMillBook}).
This class is also important in the study of Borel isomorphisms because, whenever two given Polish spaces are $\sigma$-continuously isomorphic, they are always $G_\delta$-piecewise-continuously isomorphic, whereas they are not necessarily closed-piecewise-continuously isomorphic (see \cite{KihPau,vMillBook}).
For another example, $G_\delta$-piecewise continuity is closely connected to the notion of {\em partial learning} in computational learning theory (see \cite{STLBook}).

In this article, we will introduce variations of Wadge degrees to measure the difficulty of uniformization problems.
The Wadge degrees provide a classification of subsets of a topological space with respect to continuous reducibility.
Recently, in order to analyze the structure of subsets of a higher-dimensional Polish space, several researchers started to study variations of Wadge degrees using finite-level Borel functions (see \cite{MSS15}), which are known to be related to Borel-piecewise continuous functions as mentioned above.

We will investigate subclasses of $G_\delta$-piecewise continuous reductions to compare uniformization problems which do not admit $\sigma$-continuous uniformizations.
Recall that the decomposition theorem of second-level Borel functions into $G_\delta$-piecewise continuous functions on finite dimensional Polish spaces has been proved by Semmes \cite{SemPhD}.
Remarkably, Semmes utilized a {\em priority argument} (a standard technique in computability theory) to prove his decomposition theorem on $G_\delta$-piecewise Baire-one functions.
Our ultimate goal is to understand why a priority argument is useful for analyzing $G_\delta$-piecewise continuous/Baire-one functions.

\subsection{Summary}\label{sec:summary}

In this article, the notion of $G_\delta$-piecewise continuity is subdivided into the notions of piecewise continuity with respect to {\em labeled well-founded trees}.
We will regard a labeled well-founded tree (which generates a certain subclass of the $G_\delta$-piecewise continuous functions) as a {\em priority tree}, and then function application as the act of finding the {\em true path} of the priority tree.
We will utilize this way of thinking to obtain separation results on subclasses of $G_\delta$-piecewise continuous reductions for uniformization problems on set-valued functions with compact graphs.

This method is also applicable for separating various non-constructive principles in the Weihrauch lattice.
For instance, our main results imply several statements of the following kind:
\begin{itemize}[label=($\dagger$)]
\item
For any $n\in\om$, there exist multi-valued functions $F_n,G_n:2^\om\rightrightarrows 2^\om$ whose graphs are $\Pi^0_1$ (hence $F_n,G_n\leq_{\rm W}{\sf WKL}$) such that
\begin{align*}
{\sf WKL}&\leq_{\rm W}{\sf C}_\mathbb{N}\star({\sf LPO}')^\ast\star\dots\star{\sf C}_\mathbb{N}\star{\sf LPO}'\star F_n,\\
{\sf WWKL}&\not\leq_{\rm W}({\sf LPO}')^\ast\star{\sf C}_\mathbb{N}\star\dots\star({\sf LPO}')^\ast\star{\sf C}_\mathbb{N}\star F_n,\\
{\sf WKL}&\leq_{\rm W}{\sf C}_\mathbb{N}\star({\sf LPO}')^\ast\star\dots\star{\sf C}_\mathbb{N}\star({\sf LPO}')^\ast\star{\sf LPO}\star G_n,\\
{\sf WWKL}&\not\leq_{\rm W}({\sf LPO}')^\ast\star{\sf C}_\mathbb{N}\star\dots\star({\sf LPO}')^\ast\star{\sf C}_\mathbb{N}\star({\sf LPO}')^\ast\star G_n.
\end{align*}
Here, ${\sf A}\star{\sf B}\star\cdots$ indicates $n$ repetitions of the sequential composition ${\sf A}\star {\sf B}$, i.e., $({\sf A}\star{\sf B})^{(n)}$.
The symbols ${\sf WKL}$, ${\sf WWKL}$, ${\sf LPO}$, and ${\sf C}_\mathbb{N}$ denote weak K\"onig's lemma, weak weak K\"onig's lemma, the limited principle of omniscience, and the closed choice principle on the natural numbers, respectively.
Moreover, $\star$, ${}^\ast$, and $'$ denote sequential composition, finite parallelization, and the jump operation, respectively.
\end{itemize}

\noindent For notations and terminologies in the above statement ($\dagger$), see \cite{BGR15,BGM12,BraPau12}.
We will not use any of the above notations and terminologies in the proof of our main theorems, so we do not require that the reader be familiar with the Weihrauch lattice.

\subsection{Notations}

Let $\om$ denote the set of all non-negative integers.
For a set $X$, by $X^{<\om}$ we mean the set of all finite strings $\sigma$ from $X$, that is, all functions $\sigma$ whose domain is a finite initial segment of $\om$ such that $\sigma(n)\in X$ for all $n\in{\rm dom}(\sigma)$.
This ${\rm dom}(\sigma)$ is also written as $|\sigma|$, and called the length of $\sigma$.
An individual string $\sigma\in X^{<\om}$ is sometimes written as $\langle\sigma(0),\sigma(1),\dots,\sigma(|\sigma|-1)\rangle$.
In particular, the empty string is denoted by $\langle\rangle$.
For strings $\sigma,\tau\in X^{<\om}$, by $\sigma\fr \tau$ we denote the concatenation of $\sigma$ and $\tau$.
By $\sigma^{\rm last}$ we denote the last entry of $\sigma$, and then $\sigma^-$ is the result by dropping the last entry from $\sigma$, that is, $\sigma=\sigma^-\fr\sigma^{\rm last}$.
We write $\sigma\preceq\tau$ if $\sigma$ is an initial segment of $\tau$, and if $n<|\sigma|$ then, by $\sigma\upr n$ we denote the unique initial segment of $\sigma$ of length $n$.
A tree $T$ on $X$ is a subset of $X^{<\om}$ closed under taking initial segments.
The unique $\preceq$-minimal element (that is, the empty string $\langle\rangle$) of a tree $T$ is called the root.
A string $\sigma\in T$ is a terminal or a leaf if it is a $\preceq$-maximal node.
By $T^{\rm leaf}$ we denote the set of all leaves in $T$.
For each $\sigma\in T$, by ${\rm succ}_T(\sigma)$ we denote the set of all immediate successors of $\sigma$.

We also use several notions and techniques from Computability Theory.
For instance, by $\leq_T$ we denote Turing reducibility, and for $x,y\in X^\om$, the sum $x\oplus y$ is defined by $(x\oplus y)(2n)=x(n)$ and $(x\oplus y)(2n+1)=y(n)$ for each $n\in\om$.
For basic terminology from Computability Theory and Computable Analysis, see Soare \cite{SoareBook} and Weihrauch \cite{WeihBook}, respectively.

\section{Borel-Piecewise Continuous Reducibility}

\subsection{Uniformization Problems}

In this article, a {\em space} is always assumed to be separable metrizable.
For spaces $\mathcal{X}$ and $\mathcal{Y}$, a relation $F\subseteq \mathcal{X}\times\mathcal{Y}$ is called a (partial) {\em set-valued function} or a (partial) {\em multi-valued function}, and denoted by $F:\subseteq\mathcal{X}\rightrightarrows\mathcal{Y}$.
We also denote by $F(x)$ the set $\{y\in\mathcal{Y}:(x,y)\in F\}$, and by ${\rm dom}(F)$ the set $\{x\in\mathcal{X}:F(x)\not=\emptyset\}$.
A function $\psi:{\rm dom}(F)\to\mathcal{Y}$ is called a {\em selection} or a {\em uniformization} of $F$ if $\psi(x)\in F(x)$ for all $x\in{\rm dom}(F)$.
In this case, we say that {\em $\psi$ uniformizes $F$}.

There are numerous works on uniformization theorems (measurable/continuous selection theorems; see \cite{JR02,Wag77,Wag80}).
For instance, it is known that every Borel relation in a product Polish space admits a uniformization which is measurable with respect to the smallest $\sigma$-algebra including all analytic sets (Yankov-von Neumann), while such a set does not necessarily admit a Borel uniformization (Novikov).
Recently, the classification problem of individual uniformization problems in classical mathematics has started to be developed in Computable Analysis (see \cite{BraGui11b,BGR15,BGM12}) based on ideas originated from Reverse Mathematics.

In this article, we focus on uniformization problems on compact-valued functions.
Indeed, we require set-valued functions not only to be compact-valued, but also to have compact graphs.
We sometimes call such a function a {\em compact-graph multifunction}.
It is known that such a function is {\em upper semi-continuous}.
Selection/uniformization problems on upper semi-continuous closed-valued functions have been widely investigated in various areas of mathematics (see Jayne-Rogers \cite{JR02}).
In particular, one can deduce the following result from known facts:

\begin{fact}
A compact set $\mathcal{K}\subseteq 2^\om\times 2^\om$ always admits a Baire-one uniformization, wheareas $\mathcal{K}$ does not necessarily admit a $\sigma$-continuous uniformization.
\end{fact}

The above fact can also be obtained from the Kleene-Kreisel Basis Theorem, and the Kleene Non-Basis Theorem (see Kihara \cite{Kih15} for how to interpret the results from Computability Theory in the context of $\sigma$-continuity; see also Section \ref{section:main-computable}).
Numerous number of results concerning compact-graph multifunctions on $2^\om$ which does not admit $\sigma$-continuous uniformizations are known in Computability Theory.
Here are some examples:

\begin{example}\label{example:non-uniformizable}
\hfill
\begin{enumerate}
\item There is a $\mu$-positive compact set in a probability space $(\mathcal{X},\mu)$ which does not admit a $\sigma$-continuous uniformization.
For instance, 
\[\{(x,y)\in 2^\om\times 2^\om:K^x(y\upr n)\geq n-1\}\]
 is such a set with respect to the product measure obtained by fair coin tossing, where $K^x(\sigma)$ is the prefix-free Kolmogorov complexity of a binary string $\sigma$ relative to an oracle $x$ (that is, $K^x(\sigma)$ is the length of a shortest program in some fixed programming language describing the string $\sigma$ with the help of the oracle $x$; see \cite{NiesBook}).
See also Brattka-Gherardi-H\"olzl \cite{BGR15}.
\item There is a compact set $\mathcal{K}\subseteq 2^\om\times[0,1]^2$ such that $\mathcal{K}(x)$ is a nonempty contractible dendroid (arcwise connected hereditarily unicoherent continuum) for any $x\in 2^\om$ which does not admit a $\sigma$-continuous uniformization.
See Kihara \cite{Kihara12b}.
\end{enumerate}
\end{example}

\noindent There are also a large number of interesting examples of compact-graph multifunctions on $2^\om$ which admit $\sigma$-continuous uniformizations.
Note that if a compact set $\mathcal{K}\subseteq 2^\om\times 2^\om$ admits a $\sigma$-continuous uniformization, then it admits a $G_\delta$-piecewise continuous uniformization as well (see Proposition \ref{prop:HigKih}).

\begin{example}
\hfill
\begin{enumerate}
\item Let ${\sf IVT}(x)$ be the interval coded by a $\mathbf{\Pi^0_1}$-code $x\in\om^\om$ (here recall that, in descriptive set theory, we usually code a Borel set in a Polish space by using a point in Baire space $\om^\om$).
Then it is known that the set-valued function $x\mapsto{\sf IVT}(x)$ has a $\sigma$-continuous uniformization, but has no continuous uniformization.
In the context of Computable Analysis, the uniformization problem of ${\sf IVT}$ is closely related to computability-theoretic analysis of the Intermediate Value Theorem.
See \cite{BraGui11b}.
\item Given a rapidly converging Cauchy sequence $x=(q_n)_{n\in\om}\in\mathbb{Q}^\om$, let ${\sf BE}(x)$ be the set of all binary expansions of the real $r=\lim_nq_n$.
Then ${\sf BE}$ has a $\sigma$-continuous uniformization, but has no continuous uniformization.
\end{enumerate}

\noindent Note also that the above two examples admit both a Baire-one uniformization and a $\sigma$-continuous uniformization; however they do not admit a Baire-one $\sigma$-continuous uniformization.
\end{example}

\subsection{Co-Wadge Reducibility}

In this section, we propose various reducibility notions to compare {\em degrees of difficulty} of uniformization problems.
There are several natural ways of introducing a notion of reducibility among uniformization problems, e.g., one can adopt Wadge reducibility and Weihrauch reducibility for this purpose.
In this article, we will combine these reducibility notions with Borel-piecewise continuity.
Let $\mathcal{K}$ be a class of functions, e.g., continuous functions, $G_\delta$-piecewise continuous functions, and $\sigma$-continuous functions.
Here we assume that $\mathcal{K}$ absorbs continuous (or computable) functions in the sense that for any continuous (or computable) functions $\varphi$ and $\psi$, if $\theta$ is a $\mathcal{K}$-function, then so is $x\mapsto \varphi(x,\theta\circ \psi(x))$.
In other words, $\mathcal{K}$ forms a lower cone in the (continuous) Weihrauch degrees.

For two subsets $A,B\subseteq\mathcal{X}$ of a topological space $\mathcal{X}$, we say that $A$ {\em is $\mathcal{K}$-Wadge reducible to} $B$ if there is a $\mathcal{K}$-function $\theta:\mathcal{X}\to\mathcal{X}$ such that $A=\theta^{-1}[B]$.
If we think of a subset of $\mathcal{X}$ as a $\{0,1\}$-valued function on $\mathcal{X}$, then the equation $A=\theta^{-1}[B]$ is equivalent to $A=B\circ\theta$.
Thus, it is natural to say that for functions $f:\mathcal{X}_0\to\mathcal{Y}$ and $g:\mathcal{X}_1\to\mathcal{Y}$, {\em $f$ is $\mathcal{K}$-Wadge reducible to $g$} if there is a $\mathcal{K}$-function $\theta:\mathcal{X}_0\to\mathcal{X}_1$ such that $f=g\circ\theta$.

We further extend $\mathcal{K}$-Wadge reducibility to uniformization problems.
Let us first consider the following uniformization problem ${\sf
  Fib}(g)$ for a function $g:\mathcal{X}_1\to\mathcal{Y}$:\medskip

\begin{quote}
 Find $s:\mathcal{B}\to\mathcal{X}_1$ such that $s(y)\in g^{-1}(y)$ for all $y\in\mathcal{B}$, where $\mathcal{B}$ is the image of $\mathcal{X}_1$ under $g$.
\end{quote}\medskip

\noindent It is not hard to check that $f$ is $\mathcal{K}$-Wadge reducible to $g$ if and only if there is a $\mathcal{K}$-function $\theta:\mathcal{X}_0\to\mathcal{X}_1$ such that for any solution $s$ to ${\sf Fib}(f)$, $\theta\circ s$ is a solution to ${\sf Fib}(g)$, that is, one can show the following:

\begin{prop}\label{prop:cowadge}
For functions $f:\mathcal{X}_0\to\mathcal{Y}$ and $g:\mathcal{X}_1\to\mathcal{Y}$, $f$ is $\mathcal{K}$-Wadge reducible to $g$ if and only if there is a $\mathcal{K}$-function $\theta:\mathcal{X}_0\to\mathcal{X}_1$ such that
\[(\forall s:\mathcal{B}\to\mathcal{X}_0)\;[s\mbox{ uniformizes }f^{-1}\;\Longrightarrow\;\theta\circ s\mbox{ uniformizes }g^{-1}].\]
\end{prop}

\begin{proof}
If $f$ is $\mathcal{K}$-Wadge reducible to $g$, then there is a $\mathcal{K}$-function $\theta$ such that $f(x)=y$ if and only if $g(\theta(x))=y$ for all $x,y$; therefore $y\in f^{-1}(x)$ if and only if $y\in g^{-1}(\theta(x))$.
This $\theta$ clearly satisfies the desired condition.
Conversely, suppose that we have a $\mathcal{K}$-function $\theta$ transforming a uniformization of $f^{-1}$ into that of $g^{-1}$.
For any $x$, if $f(x)=y$ then consider a uniformization $s$ satisfying $s(y)=x$.
Then we have $\theta(s(y))=\theta(x)\in g^{-1}(y)$.
This implies $f(x)=g(\theta(x))=y$ for any $x$ and $y$.
\end{proof}

Based on this observation, for multi-valued functions $F:\mathcal{X}\rightrightarrows\mathcal{Y}_0$ and $G:\mathcal{X}\rightrightarrows\mathcal{Y}_1$, we say that {\em $F$ is $\mathcal{K}$-coWadge reducible to $G$} if there is a $\mathcal{K}$-function $\theta:\mathcal{Y}_1\to\mathcal{Y}_0$ such that
\[(\forall \psi:\mathcal{X}\to\mathcal{Y}_1)\;[\psi\mbox{ uniformizes }G\;\Longrightarrow\;\theta\circ \psi\mbox{ uniformizes }F].\]

One can also extend the notion of $\mathcal{K}$-Wadge reducibility.
We say that {\em $F:\mathcal{X}_0\rightrightarrows\mathcal{Y}$ is $\mathcal{K}$-Wadge reducible to $G:\mathcal{X}_1\rightrightarrows\mathcal{Y}$} if there is a $\mathcal{K}$-function $\theta:\mathcal{X}_0\to\mathcal{X}_1$ such that
\[(\forall \psi:\mathcal{X}_0\to\mathcal{Y})\;[\psi\mbox{ uniformizes }F\;\Longrightarrow\;\psi\mbox{ uniformizes }G\circ\theta].\]

As in the proof of Proposition \ref{prop:cowadge}, one can see the one-to-one correspondence of the dual $\mathcal{K}$-Wadge degrees and the $\mathcal{K}$-coWadge degrees:

\begin{prop}
The $\mathcal{K}$-Wadge degrees and the $\mathcal{K}$-coWadge degrees of multi-valued functions are dually isomorphic via the one-to-one correspondence $F\mapsto F^{-1}$.\qed
\end{prop}

One can also see that the $\mathcal{K}$-Wadge degrees and the $\mathcal{K}$-coWadge degrees of multi-valued functions with compact graphs are dually isomorphic as well since (the graphs of) $F$ and $F^{-1}$ are homeomorphic.
The following result states that if we restrict our attention to compact-graph multi-functions, there is no need to consider a class of functions larger than $G_\delta$-piecewise continuous functions.

\begin{prop}[see Higuchi-Kihara {\cite[Proposition 23]{HigKih14I}}]\label{prop:HigKih}
Suppose that ${F},{G}\subseteq 2^\om\times 2^\om$ are compact.
Then ${F}$ is $\sigma$-continuously coWadge reducible to ${G}$ if and only if ${F}$ is $G_\delta$-piecewise continuously coWadge reducible to ${G}$.
\qed
\end{prop}

It is also natural to consider more powerful reductions among uniformization problems.
We say that {\em $F$ is weakly $\mathcal{K}$-coWadge reducible to $G$} if there is a $\mathcal{K}$-function $k:\mathcal{X}\times\mathcal{Y}_1\to\mathcal{Y}_0$ such that
\[(\forall\psi:\mathcal{X}\to\mathcal{Y}_1)\;[\psi\mbox{ uniformizes }G\;\Longrightarrow\;k\circ\langle{\rm id},\psi\rangle\mbox{ uniformizes }F],\]
that is, $y\in G(x)$ implies $k(x,y)\in F(x)$.

\begin{prop}\label{prop:single-cowadge}
There is an order-reversing embedding of the weak $\mathcal{K}$-coWadge degrees of multi-valued functions into the $\mathcal{K}$-Wadge degrees of single-valued functions.
\end{prop}

Indeed, the weak $\mathcal{K}$-coWadge degrees of multi-valued functions are dually isomorphic to the $\mathcal{K}$-Wadge degrees of {\em trivial bundles}.
For a continuous surjection $\pi:\mathcal{E}\to\mathcal{B}$ from a topological space $\mathcal{E}$ onto another topological space $\mathcal{B}$, the triple $(\mathcal{E},\mathcal{B},\pi)$ is called a {\em bundle}.
A {\em (global) section} of a bundle $(\mathcal{E},\mathcal{B},\pi)$ is a right-inverse of $\pi$, i.e., a map $s:\mathcal{B}\to\mathcal{E}$ such that $\pi\circ s={\rm id}_\mathcal{B}$.
Note that the {\em section-finding problem} is exactly the same as the uniformization problem ${\sf Fib}(\pi)$, since $s$ is a section if and only if $s(y)\in \pi^{-1}(y)$ for all $y\in\mathcal{B}$.
For a multi-valued function $F\subseteq\mathcal{X}\times Y$, the triple $(F,{\rm dom}(F),\pi_F)$ forms a bundle, where $\pi_F(x,y)=x$ for every $(x,y)\in F$.
Such a triple is called a {\em trivial bundle}.
Note that a section of a trivial bundle $\pi_F$ corresponds to the {\em cylinderification} of a uniformization of $F$.

\begin{proof}[Proof of Proposition \ref{prop:single-cowadge}]
We claim that $F$ is weakly $\mathcal{K}$-coWadge reducible to $G$ if and only if $\pi_G$ is $\mathcal{K}$-Wadge reducible to $\pi_F$.
Let $k\in\mathcal{K}$ witness that $F$ is weakly $\mathcal{K}$-coWadge reducible to $G$.
Then we have $\pi_G(x,y)=\pi_F(x,k(x,y))$ since $(x,y)\in{\rm dom}(\pi_G)=G$, i.e., $y\in G(x)$ implies $k(x,y)\in F(x)$.
Therefore, $k_0:x\mapsto(x,k(x,y))$ witnesses that $\pi_G$ is $\mathcal{K}$-Wadge reducible to $\pi_F$.
Conversely, let $k\in\mathcal{K}$ be a $\mathcal{K}$-Wadge reduction from $\pi_G$ to $\pi_F$, and let $\psi$ be a uniformization of $G$.
Given $x$, if $\psi(x)\in G(x)$, and therefore $\pi_G(x,\psi(x))=\pi_F\circ k(x,\psi(x))=x$.
Note that $k(x,\psi(x))\in\mathcal{X}\times\mathcal{Y}_0$ where $\mathcal{X}$ and $\mathcal{Y}_0$ are the domain and the codomain of $F$, respectively.
Thus $k(x,\psi(x))$ is of the form $(k_0(x,\psi(x)),k_1(x,\psi(x)))$ and moreover $k_0(x,\psi(x))=x$ since $\pi_F\circ k(x,\psi(x))=x$.
Therefore we have $k_1(x,\psi(x))\in F(x)$, that is, $k_1$ witnesses that $F$ is weakly $\mathcal{K}$-coWadge reducible to $G$.
\end{proof}

In particular, the $\mathcal{K}$-coWadge degrees of compact-graph multifunctions are embedded into the dual of the $\mathcal{K}$-Wadge degrees of single-valued functions with compact domains via the map $F\mapsto\pi_F$.

Finally, we introduce the notion of Weihrauch reducibility, which has already been employed to classify numerous individual uniformization problems in classical analysis and related areas (see \cite{BraGui11b,BGR15,BGM12}).
Let $\mathcal{H}$ and $\mathcal{K}$ be classes of functions.
For multi-valued functions $F:\mathcal{X}_0\rightrightarrows\mathcal{X}_1$ and $G:\mathcal{Y}_0\rightrightarrows\mathcal{Y}_1$, we say that {\em $F$ is $(\mathcal{K},\mathcal{H})$-Weihrauch reducible to $G$} if and only if there are an $\mathcal{H}$-function $h:\mathcal{X}_0\to\mathcal{X}_1$ and a $\mathcal{K}$-function $k:\mathcal{X}_0\times\mathcal{Y}_1\to\mathcal{Y}_0$ such that
\[(\forall\psi:\mathcal{X}_1\to\mathcal{Y}_1)\;[\psi\mbox{ uniformizes }G\;\Longrightarrow\;k\circ\langle{\rm id},\psi\circ h\rangle\mbox{ uniformizes }F],\]
that is, $y\in G(h(x))$ implies $k(x,y)\in F(x)$.

Given a bundle $\pi:\mathcal{E}_1\to \mathcal{B}_1$ and an $\mathcal{H}$-function $h:\mathcal{B}_0\to \mathcal{B}_1$, the {\em pullback bundle} $(h^\ast\mathcal{E}_1,\mathcal{B}_0,h^\ast\pi)$ is the pullback of morphisms $\pi$ and $h$ together with the base space $\mathcal{B}_0$ and the projection $h^\ast\pi:h^\ast \mathcal{E}_1\to \mathcal{B}_0$, that is,
\[h^\ast\mathcal{E}_1=\{(x,y)\in \mathcal{B}_0\times \mathcal{E}_1:h(x)=\pi(y)\},\]
where the projection is $h^\ast\pi:(x,y)\mapsto x$.
Then, as in the proof of Proposition \ref{prop:single-cowadge}, one can see that $F$ is $(\mathcal{K},\mathcal{H})$-Weihrauch reducible to $G$ if and only if there is an $\mathcal{H}$-function $h:\mathcal{B}_0\to\mathcal{B}_1$ such that the projection $h^\ast\pi_G$ (in the pullback bundle) is $\mathcal{K}$-Wadge reducible to $\pi_F$.

One can also show analogous results of Proposition \ref{prop:HigKih} for weak $\mathcal{K}$-co-Wadge reducibility and $(\mathcal{K},\mathcal{H})$-Weihrauch reducibility, that is, there is no need to think about a class $\mathcal{K}$ of functions strictly larger than that of $G_\delta$-piecewise continuous functions; however note that it is not true for $\mathcal{H}$.

\begin{remark}
\hfill
\begin{enumerate}
\item Wadge \cite{Wadge83} introduced the notion of $\mathcal{C}$-Wadge reducibility and $\mathcal{L}$-Wadge reducibility for subsets of $\om^\om$ where $\mathcal{C}$ and $\mathcal{L}$ are the classes of continuous functions and Lipschitz functions.
The notion of $\mathcal{B}$-Wadge reducibility for Borel functions $\mathcal{B}$ is introduced by Andretta-Martin \cite{AndMar03}, and $\mathcal{B}_1^\ast$-Wadge reducibility for first-level Borel functions $\mathcal{B}_1^\ast$ (which are equivalent to closed-piecewise continuous functions by the Jayne-Rogers Theorem \cite{JR82}, and also to Baire-one-star functions) by Andretta \cite{And06b}.
For $\mathcal{K}$-Wadge reducibility with respect to other classes $\mathcal{K}$, see also Motto Ros \cite{MRos09,MRos10,MRos14} and Motto Ros-Schlicht-Selivanov \cite{MSS15}
\item The notion of $\mathcal{K}$-coWadge reducibility for various kinds of classes $\mathcal{K}$ of $\sigma$-computable functions (e.g., $\Pi^0_1$-piecewise computable functions) is first introduced by the author in his master's thesis to develop intermediate notions between Medvedev reducibility and Muchnik reducibility for mass problems, and essentially the same notion is further developed by Kihara \cite{Kihara12} and Higuchi-Kihara \cite{HigKih14I,HigKih14II}.
\item If both $\mathcal{K}$ and $\mathcal{H}$ are the sets of all computable functions, then the notion of $(\mathcal{K},\mathcal{H})$-Weihrauch reducibility is known as Weihrauch reducibility \cite{BraGui11a}, and widely studied in Computable Analysis to classify $\Pi_2$ theorems in classical mathematics \cite{BraGui11b,BGR15,BGM12}.
The notion of $(\mathcal{K},\mathcal{H})$-Weihrauch reducibility for $\mathcal{K}=\mathcal{H}=\mathcal{C}$ is also known as continuous Weihrauch reducibility.
If both $\mathcal{K}$ and $\mathcal{H}$ are the sets of all {\em $\sigma$-computable} functions (see Section \ref{section:main-computable}), then the notion of $(\mathcal{K},\mathcal{H})$-Weihrauch reducibility is known as {\em computable reducibility}, which is introduced by Dzhafarov \cite{Dzh15} (see also Hirschfeldt-Jockusch \cite{HJta}).
See also \cite{Pau15} for the category-theoretic view, and \cite{NobPau} for the relationship with the Wadge degrees.
\item 
This kind of use of a fibration is standard in categorical logic (see Jacobs \cite{Jac99}).
Especially, the above interpretation of Weihrauch reducibility in the setting of a fibration is first explicitly introduced by Yoshimura \cite{Yos1,Yos2}.
\end{enumerate}
\end{remark}

\subsection{Borel-Piecewise Continuity}

We now begin to develop a fine structure of $\sigma$-continu\-ous functions.
The notion of $\mathbf{\Gamma}$-piecewise continuity introduced in Section \ref{section:history} provides us a way of measuring the complexity of functions.
More specifically, the complexity of a $\sigma$-continuous Borel function can be defined as the least Borel complexity of a decomposition making the function be continuous.
For instance, Dirichlet's nowhere continuous function $\chi_\mathbb{Q}$ is obviously $G_\delta$-piecewise continuous, but not closed-piecewise continuous, so one can say that the Borel complexity of Dirichlet's function is exactly $2$.
One can also classify $\sigma$-continuous functions on the basis of the least cardinality of such a decomposition (see also \cite{Pau10}).
Indeed, Dirichlet's function is $G_\delta$-$2$-wise continuous, where, a function $f:X\to Y$ is {\em $\mathbf{\Gamma}$-$n$-wise continuous} if there is a $\mathbf{\Gamma}$-cover $\{X_k\}_{k<n}$ of $X$ such that $f\upr X^{\sf diff}_k$ is continuous, where $X^{\sf diff}_k=X_k\setminus\bigcup_{j<k}X_{j}$, for every $k<n$.
As another example of a $G_\delta$-$n$-wise continuous function, it is known in topological dimension theory that there is a $G_\delta$-$(n+1)$-wise embedding of $\mathbb{R}^n$ into $2^\omega$ whereas there is no $G_\delta$-$n$-wise embedding of $\mathbb{R}^n$ into $2^\omega$ (see \cite{vMillBook}).

However, this viewpoint is too coarse for our purpose.
For instance, closed-piecewise continuous functions are naturally classified in the context of the transfinite mind-change hierarchy \cite{deBre14,FreSmi93} (or equivalently, the hierarchy of Baire-one-star functions \cite{Malek06}); therefore, a decomposition of a function should be allowed to form a well-founded tree.
Indeed, Selivanov \cite{Seli07} found that, for $k\in\om$, the Wadge degrees of $\mathbf{\Delta}^0_2$-measurable $k$-valued functions $f:\om^\om\to k$ (which are indeed closed-piecewise continuous since their values have only finitely many possibilities) can be completely captured by using $k$-labeled countable forests with no infinite chains up to homomorphism.
Such a forest illustrates a dynamic process approximating a closed-piecewise continuous function $f$.
However, such a viewpoint involving a complete classification is now too complicated to analyze functions $f:\om^\om\to\om^\om$, so we here take a bit coarser standpoint.

We keep thinking about a well-founded tree illustrating a dynamic process defining a $\sigma$-continuous Borel function.
A $\mathbf{\Gamma}$-piecewise continuous function in the sense of Section \ref{section:history} is controlled by a conditional branching described by a $\mathbf{\Gamma}$ formula.
The flowchart of this control process is represented as a (possibly infinitely branching) tree $T$ of height $2$, where the root of $T$ is labeled by a $\mathbf{\Gamma}$ formula, and each leaf (terminal node) of $T$ is labeled by a partial continuous function.

\begin{example}[see Figure \ref{fig:flowchart1}]
The tree associated with Dirichlet's function turns out to be $T=\{\langle\rangle,\langle 0\rangle,\langle 1\rangle\}$, and the root node $\langle\rangle$ asks whether a given input $x\in\mathbb{R}$ is an irrational or not.
If $x$ is rational, the algorithm goes to the righthand node $\langle 1\rangle$, and returns $1$ (that is, the constant function $1$ is assigned to the node $\langle 1\rangle$), and if $x$ is irrational, go to the lefthand node $\langle 0\rangle$, and return $0$ (that is, the constant function $0$ is assigned to the node $\langle 0\rangle$).
In other words, Dirichlet's function consists of the tree $T$ and the $\Pi^0_2$ formula on the root $\langle\rangle$ described above, and two constant functions $x\mapsto 0$ and $x\mapsto 1$ on leaves $\langle 0\rangle$ and $\langle 1\rangle$ respectively.
\end{example}

\begin{figure}[t]
\centering
\includegraphics[scale=1]{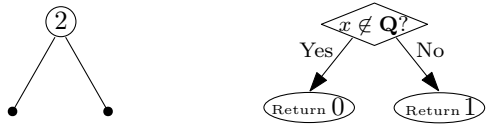}
\caption{\footnotesize (Left) A labeled well-founded tree for $G_\delta$-$2$-wise continuity; (Right) A flowchart defining Dirichlet's function $\chi_\mathbb{Q}$.}
\label{fig:flowchart1}
\end{figure}

Now it is natural to consider any well-founded tree $T\subseteq\om^{<\om}$.
Assume that each non-terminal node $\sigma\in T$ is labeled by some ordinal ${\rm rk}_T(\sigma)$ which specifies the Borel complexity of a question which can be arranged on the node $\sigma$.
Then, as before, we assign each non-terminal node $\sigma$ of $T$ to a Borel question of Borel rank ${\rm rk}_T(\sigma)$, and each terminal node of $T$ to a continuous function.
We think of this assignment on a tree as a flowchart that defines a {\em nested-piecewise continuous} function.

\begin{definition}\label{def:labeledwellfounded}
A {\em labeled well-founded tree} is a pair $\mathbf{T}=(T,{\rm rk}_T)$ of a well-founded tree $T\subseteq\om^{<\om}$ and an ordinal-valued function ${\rm rk}_T:T\to\om_1$.
A {\em flowchart} on a labeled well-founded tree $(T,{\rm rk}_T)$ is a collection $\Lambda=(P_\sigma,f_\rho)_{\sigma\in T,\rho\in T^{\rm leaf}}$ satisfying the properties (1), (2) and (3):
\begin{enumerate}
\item $P_{\pair{}}$ is a subset of $\mathcal{X}$. The set $P_{\pair{}}$ is called the {\em domain} of the flowchart $\Lambda$, and written as ${\rm dom}(\Lambda)$.
\item For every non-terminal node $\sigma\in T$, $\pair{P_\tau:\tau\in{\rm succ}_T(\sigma)}$ forms a $\mathbf{\Pi}^0_{{\rm rk}_T(\sigma)}$ cover of $P_\sigma$ (where recall that ${\rm succ}_T(\sigma)$ is the set of all immediate successors of $\sigma$ in $T$), that is, $P_\sigma\subseteq\bigcup\{P_{\sigma\fr n}:\sigma\fr n\in T\}$ and $P_{\sigma\fr n}$ is $\mathbf{\Pi}^0_{{\rm rk}_T(\sigma)}$ for every $n$.
\end{enumerate}
By the covering condition (2), we have the following property:
\[(\forall x\in{\rm dom}(\Lambda))(\exists\rho\in T^{\rm leaf})\;x\in\bigcap_{\sigma\preceq\rho}P_\sigma.\]
For every $x\in{\rm dom}(\Lambda)$, the leftmost one among such leaves $\rho$ is called the {\em true path of $\Lambda$ along $x$} and denoted by ${\rm TP}_\Lambda(x)$.
Here, we say that {\em $\sigma$ is to the left of $\tau$} (written as $\sigma\leq_{\rm left}\tau$) if either $\sigma=\tau$ or there is $n$ such that $\sigma\upr n=\tau\upr n$ but $\sigma(n)<\tau(n)$.
Then we define $D_\rho$ as the set of all $x\in{\rm dom}(\Lambda)$ such that ${\rm TP}_\Lambda(x)=\rho$.

\begin{enumerate}
\item[(3)] $f_\rho:D_\rho\to\mathcal{Y}$ is a continuous function with domain $D_\rho$ for every leaf $\rho\in T^{\rm leaf}$.
\end{enumerate}\medskip

A flowchart $\Lambda$ always defines a function $f_\Lambda:{\rm dom}(\Lambda)\to\mathcal{Y}$ as follows:
\[f_\Lambda(x)=f_{{\rm TP}_\Lambda(x)}(x).\]
\end{definition}\medskip

Intuitively, a flowchart $\Lambda$ on a labeled well-founded tree $T$ describes a (non-effective) algorithm defining the function $f_\Lambda$ as follows:
Given an input $x$, if the algorithm reaches a non-terminal node $\sigma\in T$, the flowchart $\Lambda$ asks the following:
\[\mbox{What is the {\em least} $n$ such that $x\in P_{\sigma\fr n}$?}\]
Although this question is not necessarily computably decidable, our ``algorithm'' following $\Lambda$ is allowed to be non-effective, and so always answers to this question by the correct value $n$.
Then the algorithm moves to the node $\sigma\fr n$ for such $n$.
If the algorithm reaches a terminal node $\rho\in T^{\rm leaf}$ (that is, $\rho$ is the true path of $\Lambda$ along $x$), then it returns the output $f_\rho(x)$.

\begin{definition}
Let $\mathbf{T}=(T,{\rm rk}_T)$ be a labeled well-founded tree.
A function $f:\mathcal{X}\to\mathcal{Y}$ is {\em $\mathbf{T}$-piecewise continuous} if there is a flowchart $\Lambda$ on $\mathbf{T}$ such that $f=f_\Lambda$.
\end{definition}

For a countable ordinal $\xi$, a labeled well-founded tree $\mathbf{T}=(T,{\rm rk}_T)$ is of {\em Borel rank $(\xi,\eta)$} if ${\rm rk}_T(\sigma)\leq \xi$ for all infinitely branching nodes $\sigma\in T$, and ${\rm rk}_T(\sigma)\leq\eta$ for all finitely branching nodes $\sigma\in T$.
We mainly focus on labeled well-founded trees $\mathbf{T}$ of Borel rank $(1,2)$.
It is clear that $\mathbf{T}$-piecewise continuity implies $G_\delta$-piecewise continuity whenever $\mathbf{T}$ is of Borel rank $(1,2)$.

\begin{remark}
One can also introduce piecewise continuity on a labeled directed graph (which can be represented by a labeled tree $T$ not necessarily well-founded) by declaring that $f_\Lambda(x)$ is undefined whenever the algorithm following $\Lambda$ on an input $x$ never reaches a halting state (i.e., a leaf of $T$).
This generalization seems natural in computer science, e.g., a Blum-Shub-Smale (BSS) machine \cite{BSS89} has a power to answer to a noncomputable $\Pi^0_1$ question of the form ``$x=0$?'', and the original definition of a BSS computation is clearly given by a flowchart on a labeled directed graph.
See also Neumann-Pauly \cite{NePa16}.
\end{remark}

\subsection*{Weihrauch Reduction}

Given a labeled well-founded tree $\mathbf{T}$, one can define the associated reducibility notions by using $\mathbf{T}$-piecewise continuity.
By $\mathbf{T}\mathcal{C}$ we denote the class of $\mathbf{T}$-piecewise continuous functions, and we often omit the symbol $\mathcal{C}$ when we mention associated reducibility notions, e.g., $\mathbf{T}\mathcal{C}$-Wadge reducibility is often abbreviated as $\mathbf{T}$-Wadge reducibility.
Here the class of $\mathbf{T}$-piecewise continuous functions is not necessarily closed under composition; therefore, for instance, the $\mathbf{T}$-Wadge ordering may not be transitive.
If we hope to recover transitivity of the associated orderings, we have to consider a collection of labeled well-founded trees.
However, even if it does not satisfy transitivity, the understanding of the associated reducibility still has a consequence in the context of Weihrauch degrees.

Let us first consider the labeled well-founded tree $\mathbf{T}_{\xi,2}=(T,{\rm rk}_T)$ defining $\mathbf{\Pi}^0_\xi$ $2$-wise continuity, that is, $T=\{\langle\rangle,\langle 0\rangle,\langle 1\rangle\}$ and ${\rm rk}_T(\langle\rangle)=\xi$.
One may notice that $\mathbf{T}_{1,2}$-piecewise continuity (i.e., closed $2$-wise continuity) has some connection with the {\em limited principle of omniscience} (the {\em law of excluded middle for $\Sigma^0_1$ formulas}) in constructive analysis.
In the context of Weihrauch degrees, the limited principle of omniscience is interpreted by the function ${\sf LPO}:\om^\om\to 2$ defined by ${\sf LPO}(x)=0$ if $x(n)=0$ for some $n\in\om$; otherwise ${\sf LPO}(x)=1$.
Clearly ${\sf LPO}$ is closed $2$-wise continuous, and conversely, it is not hard to check that every closed $2$-wise continuous function $g$ is of the form $k\circ\langle{\rm id},{\sf LPO}\circ h\rangle$ for some continuous functions $h,k$, that is, $g$ is continuously Weihrauch reducible to ${\sf LPO}$.
We generalize this observation to any labeled well-founded tree.

\begin{prop}\label{prop:chara-label}
Let $\mathbf{T}$ be a labeled well-founded tree.
For a single-valued function $f$, the following are equivalent:
\begin{enumerate}
\item $f$ is $\mathbf{T}$-piecewise continuous.
\item $f$ is coWadge reducible to ${\rm TP}_\Lambda$ for some flowchart $\Lambda$ on $\mathbf{T}$.
\item $f$ is continuously Weihrauch reducible to ${\rm TP}_\Lambda$ for some flowchart $\Lambda$ on $\mathbf{T}$.
\end{enumerate}
\end{prop}

\begin{proof}
To see the implication (1)$\Rightarrow$(2), assume that $f$ is $\mathbf{T}$-piecewise continuous.
Then there is a flowchart $\Lambda=(P_\sigma,f_\rho)$ on $\mathbf{T}$ such that $f=f_\Lambda$.
Define $k(x,\rho)=f_\rho(x)$.
Note that $k$ is continuous since each $f_\rho$ is continuous and $T^{\rm leaf}$ is countable.
It is not hard to see that $f(x)=k(x,{\rm TP}_\Lambda(x))$.
The implication from (2) to (3) is obvious.
To see the implication (3)$\Rightarrow$(1), we assume that $f$ is continuously Weihrauch reducible to ${\rm TP}_\Lambda$ for some flowchart $\Lambda=(P_\sigma,f_\rho)$ on $\mathbf{T}$, that is, there are continuous functions $h,k$ such that $f(x)=k(x,{\rm TP}_\Lambda(h(x)))$ for all $x$.
Define $g_\rho(x)=k(x,\rho)$ and consider the flowchart $\Lambda^\ast=(h^{-1}[P_\sigma],g_\rho)$.
Note that the continuous preimage does not increase the Borel complexity of a set; therefore $\Lambda^\ast$ is a flowchart on $\mathbf{T}$.
It is not hard to see that ${\rm TP}_{\Lambda^\ast}(x)={\rm TP}_\Lambda(h(x))$.
Consequently, $f_{\Lambda^\ast}(x)=k(x,{\rm TP}_\Lambda(h(x)))=f(x)$ as desired.
\end{proof}

One can also show the similar result for multi-valued functions.
For multi-valued functions $F$ and $G$, the composition $G\circ F$ is defined as follows:
\begin{align*}
{\rm dom}(G\circ F)&=\{x\in{\rm dom}(F):F(x)\subseteq{\rm dom}(G)\},\\
y\in G\circ F(x)\;&\iff\;(\exists z)\;[z\in F(x)\mbox{ and }y\in G(z)].
\end{align*}
Then, the {\em sequential composition} $G\star F$ (see \cite{BGM12,BraPau12}) is defined as a multi-valued function realizing the greatest Weihrauch degree among those of multi-valued functions of the form $G_0\circ F_0$ such that $G_0$ and $F_0$ are Weihrauch reducible to $G$ and $F$, respectively.

\begin{prop}\label{prop:weihrauch}
Let $\mathbf{T}$ be a labeled well-founded tree, and let $F$ and $G$ be multi-valued functions.
Then, the following are equivalent:
\begin{enumerate}
\item $F$ is $(\mathbf{T},\mathcal{C})$-Weihrauch reducible to $G$.
\item $F$ is continuously Weihrauch reducible to ${\rm TP}_\Lambda\star G$ for some flowchart $\Lambda$ on $\mathbf{T}$.
\end{enumerate}
\end{prop}

\begin{proof}
Assume that $F$ is $(\mathbf{T},\mathcal{C})$-Weihrauch reducible to $G$.
Then there are a $\mathbf{T}$-piecewise continuous function $k$ and a continuous function $h$ such that for any $x$, whenever $y\in G(h(x))$, we have $k(x,y)\in F(x)$.
We consider $h_0(x)=(x,h(x))$ and $G_0(x,y)=\{x\}\times G(y)=\{(x,z):z\in G(y)\}$.
Clearly $h_0$ is continuous, and $G_0$ is Weihrauch reducible to $G$.
Moreover, by Proposition \ref{prop:chara-label}, $k$ is continuously Weihrauch reducible to ${\rm TP}_\Lambda$ for some flowchart $\Lambda$ on $\mathbf{T}$.
We claim that $F$ is Weihrauch reducible to $k\circ G_0$ via $k_0={\rm id}$ and $h_0$, that is, $z\in k\circ G_0(h_0(x))$ implies $z\in F(x)$.
We note that $G_0(h_0(x))=G_0(x,h(x))=\{x\}\times G(h(x))$.
Therefore, if $z\in k\circ G_0(h_0(x))$, then there is $y\in G(h(x))$ such that $z=k(x,y)$.
Then, by our choice of $h$ and $k$, we have $z\in F(x)$ as desired.

Conversely, assume that $F$ is continuously Weihrauch reducible to ${\rm TP}_\Lambda\star G$ for some flowchart $\Lambda$ on $\mathbf{T}$, that is, there are $k^\ast$ and $G^\ast$ such that $k^\ast$ is Weihrauch reducible to ${\rm TP}_\Lambda$ and $G^\ast$ is Weihrauch reducible to $G$, and moreover there are continuous functions $h_1,k_1$ such that for any $x$, the condition $y\in k^\ast\circ G^\ast(h_1(x))$ implies $k_1(x,y)\in F(x)$.
Note that $k^\ast$ can be assumed to be single-valued since ${\rm TP}_\Lambda$ is single-valued, and therefore, a Weihrauch reduction to ${\rm TP}_\Lambda$ gives a uniformization of $k^\ast$.
Thus, every $y\in k^\ast\circ G^\ast(h_1(x))$ is of the form $k^\ast(z)$ for some $z\in G^\ast(h_1(x))$.
Therefore, $z\in G^\ast(h_1(x))$ implies $k_1(x,k^\ast(z))\in F(x)$.
Let $k_2$ and $h_2$ be continuous functions witnessing that $G^\ast$ is Weihrauch reducible to $G$.
Then we get that $y\in G(h_2\circ h_1(x))$ implies $k_1(x,k^\ast\circ k_2(h_1(x),y))\in F(x)$.
Since $k_1$, $k_2$, and $h_1$ are continuous, it is clear that the function $k^{\ast\ast}$ defined by $k^{\ast\ast}(x,y)=k_1(x,k^\ast\circ k_2(h_1(x),y))$ is continuously Weihrauch reducible to $k^\ast$.
By Proposition \ref{prop:chara-label}, $k^{\ast\ast}$ is $\mathbf{T}$-piecewise continuous, and therefore, $F$ is $(\mathbf{T},\mathcal{C})$-Weihrauch reducible to $G$ via $k^{\ast\ast}$ and $h_2\circ h_1$.
\end{proof}

\begin{example}\label{example:Weihrauch-principle}
By Proposition \ref{prop:weihrauch} and by the previous discussion, $F$ is $(\mathbf{T}_{1,2},\mathcal{C})$-Weihrauch reducible to $G$ if and only if $F$ is continuously Weihrauch reducible to ${\sf LPO}\star G$.
We also have similar connections between $\mathbf{T}_{1,\om}$ (closed-piecewise continuity) and the closed choice principle $\mathsf{C}_\mathbb{N}$ on the natural numbers, and between $\mathbf{T}_{2,2}$ ($G_\delta$-$2$-wise continuity) and the jump ${\sf LPO}'$ of ${\sf LPO}$ (see \cite{BGM12} for the jump of a multi-valued function).
\end{example}

Later, for a given suitable collection $\mathbb{V}$ of labeled well-founded trees, we will construct a labeled well-founded tree $\mathbf{T}(\mathbb{V}')$ defining a class of piecewise continuous functions not much larger than the class defined by $\mathbb{V}$.
By using the relationship between piecewise continuity and sequential composition obtained from Proposition \ref{prop:weihrauch}, we will prove a separation result of the following form:
Given suitable $\Pi^0_1$ uniformization problems $\mathcal{S}$ and $\mathcal{U}$ on $2^\om$ which do not admit $\sigma$-continuous uniformizations, one can construct another $\Pi^0_1$ uniformization problem $\mathcal{T}$ on $2^\om$ such that
\begin{enumerate}
\item $\mathcal{S}$ is Weihrauch reducible to ${\rm TP}_{\Lambda'}\star\mathcal{T}$ for some flowchart $\Lambda'$ on $\mathbf{T}(\mathbb{V}')$.
\item $\mathcal{U}$ is not Weihrauch reducible to ${\rm TP}_{\Lambda}\star\mathcal{T}$ for any flowchart $\Lambda$ on $\mathbf{T}\in\mathbb{V}$.
\end{enumerate}

\subsection*{Vein-Piecewise Continuity}\label{section:vein-piecewise}

Hereafter, we will not care about the number of branches of a finitely branching node of a labeled well-founded tree $\mathbf{T}$, that is, we will only specify the type of a node: a leaf, a finitely branching node, or an infinitely branching node.
For instance, consider $\mathbf{\Pi}^0_\xi$-finite-piecewise-continuity, where we say that a function is {\em $\mathbf{\Pi}^0_\xi$-finite-piecewise continuous} if it is $\mathbf{\Pi}^0_\xi$-$k$-wise continuous for some $k$.
Such a notion corresponds to the countable collection $\mathbf{T}_{\xi,<\om}=\{\mathbf{T}_{\xi,k}:k\in\om\}$ such that $\mathbf{T}_{\xi,k}$ corresponds to $\mathbf{\Pi}^0_\xi$-$k$-wise continuity, i.e., a tree of height $2$ whose root is $k$-branching and labeled by $\xi$.

We introduce a single tree $\mathbb{V}_\xi$ (called a {\em vein}) generating the collection $\mathbf{T}_{\xi,<\om}$.
Let us think of $0$, $1$, and $\om$ just as symbols indicating that it is a ``leaf'', ``{\em finitely branching}'', and ``infinitely branching'', respectively.
In particular, we treat a $1$-branching node (that is, a non-terminal non-branching node) as if it were a finitely branching node.
Let $\mathbb{V}_\xi=\{\langle\rangle,\langle\ast\rangle\}$ be the tree whose root is labeled by $\xi$.
The root is $1$-branching, so it is ``finitely branching'', and moreover labeled by $\xi$.
Thus, we regard that $\mathbb{V}_\xi$ represents $\mathbf{\Pi}^0_\xi$-finite-piecewise continuity.

We now introduce the formal definition.
For a tree $T\subseteq\om^{<\om}$ and a string $\sigma\in T$, by ${\rm br}_T(\sigma)$ we denote the number of immediate successors of $\sigma$ in $T$.

\begin{definition}
A {\em vein} is a labeled well-founded tree $\mathbb{V}=(\mathcal{V},{\rm rk}_\mathcal{V})$ such that ${\rm br}_\mathcal{V}(\sigma)\in\{0,1,\om\}$ for every $\sigma\in\mathcal{V}$.
\end{definition}

The intended meaning of this notion is that a vein $\mathbb{V}$ is not only a labeled well-founded tree, but also represents the smallest collection of labeled well-founded trees including $\mathbb{V}$ itself and closed under any transformation which converts a non-branching non-terminal node $\sigma\in\mathcal{V}$ into a finitely-branching node whose successors are copies of successors of $\sigma$.
That is, the equation ${\rm br}_\mathcal{V}(\sigma)=1$ indicates that $\sigma\in\mathcal{V}$ is a finitely-branching node in $\mathcal{V}$, but the number of the immediate successors of $\sigma$ can be any finite value.
For instance, we identify the above $\mathbb{V}_\xi$ with $\mathbf{T}_{\xi,<\om}$.

\begin{figure}[t]
\centering
\includegraphics[scale=0.8]{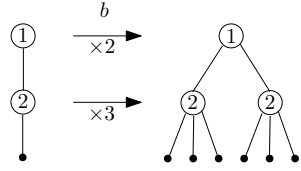}
\caption{\footnotesize The $b$-branching of the vein $\mathcal{V}=\{\langle\rangle,\langle\ast\rangle,\langle\ast,\ast\rangle\}$, where $b(\langle\rangle)=2$ and $b(\langle\ast\rangle)=3$.}
\label{fig:flowchart2}
\end{figure}

We give the formal definition of the above idea.
By $\mathcal{V}^{\sf fin}$ we denote the set of all non-terminal strings $\sigma$ such that ${\rm br}_\mathcal{V}(\sigma)<\om$.
A {\em branching function} is a function $b:\mathcal{V}^{\sf fin}\to\om$.
The role of this function is to convert each node $\sigma\in\mathcal{V}$ with ${\rm br}_\mathcal{V}(\sigma)=1$ into the $b(\sigma)$-branching node whose successors are copies of successors of $\sigma$ as mentioned above (see Figure \ref{fig:flowchart2}).
Given a vein $\mathbb{V}=(\mathcal{V},{\rm rk}_\mathcal{V})$ and a branching function $b:\mathcal{V}^{\sf fin}\to\om$, we inductively define $\mathbb{V}^b=(\mathcal{V}^b,{\rm rk}_{\mathcal{V}^b})$,  the {\em $b$-branching of $\mathbb{V}$} (see also Figure \ref{fig:flowchart2}),  with a {\em copy-source-referring function} $\iota:\mathcal{V}^b\to \mathcal{V}$ as follows:
\begin{enumerate}
\item $\pair{}\in \mathcal{V}^b$ and $\iota(\pair{})=\pair{}$.
\item If $\sigma\in \mathcal{V}^b$ and ${\rm br}_\mathcal{V}(\iota(\sigma))=1$, then $\sigma$ is converted into a $b(\iota(\sigma))$-branching node, that is,
\begin{align*}
\pair{{\rm rk}_{\mathcal{V}^b}(\sigma),{\rm br}_{\mathcal{V}^b}(\sigma)}=&\pair{{\rm rk}_\mathcal{V}(\iota(\sigma)),b(\iota(\sigma))},\\
\sigma\fr n\in \mathcal{V}^b, \quad\quad & \iota(\sigma\fr n)=\iota(\sigma)\fr\ast,
\end{align*}
for every $n<b(\sigma)$.
Here $\iota(\sigma)\fr\ast$ is the unique immediate successor of $\iota(\sigma)$ in $\mathcal{V}$.
\item If $\sigma\in \mathcal{V}^b$ and ${\rm br}_\mathcal{V}(\iota(\sigma))=\om$, then $\sigma$ remains the same as the copy source node $\iota(\sigma)$, that is,
\begin{align*}
\pair{{\rm rk}_{\mathcal{V}^b}(\sigma),{\rm br}_{\mathcal{V}^b}(\sigma)}=&\pair{{\rm rk}_\mathcal{V}(\iota(\sigma)),\om},\\
\sigma\fr n\in \mathcal{V}^b, \quad\quad & \iota(\sigma\fr n)=\iota(\sigma)\fr n,
\end{align*}
for every $n$ such that $\iota(\sigma)\fr n\in \mathcal{V}$.
\item If $\sigma\in \mathcal{V}^b$ and ${\rm br}_\mathcal{V}(\iota(\sigma))=0$, then $\sigma$ remains the same as the copy source node $\iota(\sigma)$, that is,
\begin{align*}
\pair{{\rm rk}_{\mathcal{V}^b}(\sigma),{\rm br}_{\mathcal{V}^b}(\sigma)}=&\pair{{\rm rk}_\mathcal{V}(\iota(\sigma)),0}.
\end{align*}
\end{enumerate}\medskip

\noindent We often identify $\mathbb{V}$ with the collection $(\mathbb{V}^b\mid b:\mathcal{V}^{\sf fin}\to \om)$.
A {\em flow} on a vein $\mathbb{V}$ is a pair $\Lambda=(\mathbf{T},\Gamma)$ such that $\Gamma$ is a flowchart on a labeled well-founded tree of the form $\mathbf{T}=\mathbb{V}^b$ for some branching function $b:\mathcal{V}^{\sf fin}\to\om$.
A flow $\Lambda=(\mathbf{T},\Gamma)$ automatically induces a function $f_\Lambda$ as in Definition \ref{def:labeledwellfounded}, that is, $f_\Lambda=f_\Gamma$.

\begin{definition}
Let $\mathbb{V}=(\mathcal{V},{\rm rk}_\mathcal{V})$ be a vein.
A function $f:\mathcal{X}\to\mathcal{Y}$ is {\em $\mathbb{V}$-piecewise continuous} if there is a flow $\Lambda$ on the vein $\mathbb{V}$ such that $f=f_\Lambda$.
\end{definition}

A vein $\mathbb{V}$ is {\em of Borel rank $(\zeta,\eta)$} if it is of Borel rank $(\zeta,\eta)$ as a labeled well-founded tree.

\subsection*{Operations on Veins}

We first note that $\mathbb{V}$-piecewise continuity may not be closed under taking composition.
Therefore, it is natural to introduce the notion of a {\em transitive closure} of a vein $\mathbb{V}$.
Given two veins $\mathbb{V}_0=(\mathcal{V}_0,{\rm rk}_{\mathcal{V}_0})$ and $\mathbb{V}_1=(\mathcal{V}_1,{\rm rk}_{\mathcal{V}_1})$, the {\em concatenation} $\mathbb{V}_0\fr\mathbb{V}_1=(\mathcal{V}_0\fr\mathcal{V}_1,{\rm rk}_{\mathcal{V}_0\fr\mathcal{V}_1})$ of $\mathbb{V}_0$ and $\mathbb{V}_1$ is defined as follows:
\begin{align*}
\mathcal{V}_0\fr\mathcal{V}_1&=\{\sigma:(\exists\rho\in \mathcal{V}_0^{\rm leaf})(\exists\tau\in \mathcal{V}_1)\;\sigma\preceq\rho\fr \tau\},\\
{\rm rk}_{\mathcal{V}_0\fr\mathcal{V}_1}(\sigma)&=
\begin{cases}
{\rm rk}_{\mathcal{V}_0}(\sigma)&\mbox{ if }\sigma\in\mathcal{V}_0\setminus\mathcal{V}_0^{\rm leaf},\\
{\rm rk}_{\mathcal{V}_1}(\tau)&\mbox{ if }\sigma=\rho\fr\tau\mbox{ for some }\rho\in\mathcal{V}_0^{\rm leaf}.
\end{cases}
\end{align*}

If $f_i$ is $\mathbb{V}_i$-piecewise continuous for each $i<2$, the composition $f_1\circ f_0$ is obviously $(\mathbb{V}_0\fr\mathbb{V}_1)$-piecewise continuous.
Define $\mathbb{V}^{(1)}=\mathbb{V}$, and $\mathbb{V}^{(n+1)}=\mathbb{V}^{(n)}\fr\mathbb{V}$ for each $n\in\om$.
Then, we can think of the countable collection ${\rm trcl}(\mathbb{V}):=(\mathbb{V}^{(n)})_{n\in\om}$ as the {\em transitive closure} of $\mathbb{V}$.

Next, it is also worth mentioning that every vein (indeed, every countable collection of veins) is dominated by a single labeled well-founded tree.
Here we say that for collections $\mathbb{V}_0,\mathbb{V}_1$ of labeled well-founded trees, $\mathbb{V}_1$ {\em dominates} $\mathbb{V}_0$ if every $\mathbb{V}_0$-piecewise continuous function is $\mathbb{V}_1$-piecewise continuous (recall that every vein $\mathbb{V}$ is identified with the collection $(\mathbb{V}^b\mid b:\mathcal{V}^{\sf fin}\to\om)$ of labeled well-founded trees). 
Given a vein $\mathbb{V}=(\mathcal{V},{\rm rk}_\mathcal{V})$ we inductively define the {\em closure} $\overline{\mathbb{V}}=(\overline{\mathcal{V}},{\rm rk}_{\overline{\mathcal{V}}})$ of $\mathbb{V}$ (with a copy-source-referring function $\iota:\overline{\mathcal{V}}\to\mathcal{V}$) as follows:

\begin{enumerate}
\item $\pair{}\in \overline{\mathcal{V}}$ and $\iota(\pair{})=\pair{}$.
\item If $\sigma\in \overline{\mathcal{V}}$ and ${\rm br}_\mathcal{V}(\iota(\sigma))=1$, then we insert a new infinite branch of Borel complexity $0$ whose successors are copies of successors of $\iota(\sigma)$, that is,
\begin{align*}
\pair{{\rm rk}_{\overline{\mathcal{V}}}(\sigma),{\rm br}_{\overline{\mathcal{V}}}(\sigma)}=&\pair{0,\om},\\
\pair{{\rm rk}_{\overline{\mathcal{V}}}(\sigma\fr n),{\rm br}_{\overline{\mathcal{V}}}(\sigma\fr n)}=&\pair{{\rm rk}_\mathcal{V}(\iota(\sigma)),{\rm br}_\mathcal{V}(\iota(\sigma))},\\
\sigma\fr n,\sigma\fr n\fr\ast\in\overline{\mathcal{V}}, \quad\quad & \iota(\sigma\fr n\fr\ast)=\iota(\sigma)\fr\ast,
\end{align*}
for every $n\in\om$, where $\sigma\fr\ast$ is the unique immediate successor of $\iota(\sigma)$ in $\mathcal{V}$
\item If $\sigma\in \overline{\mathcal{V}}$ and ${\rm br}_\mathcal{V}(\iota(\sigma))\in\{0,\om\}$, then $\sigma$ remains unchanged.
\end{enumerate}

We define a branching function $d:\overline{\mathcal{V}}{}^{\sf fin}\to \om$ by $d(\sigma)=\sigma^{{\rm last}}$ for every nonempty string $\sigma\in\overline{\mathcal{V}}{}^{\sf fin}$.
Here, note that if $\iota(\sigma)=\tau$ is a finitely branching node in $\mathbb{V}$, then there are infinitely many copies of $\tau$ below $\sigma$ in the closure $\overline{\mathbb{V}}$.
The branching function $d$ converts the $n$-th such copy into an $n$-branching node, that is, $d$ realizes our intended meaning of a finitely-branching node $\tau$ of a vein $\mathbb{V}$ that the number of branches of $\tau$ can be any finite value.
Then, we consider the labeled well-founded tree $\mathbf{T}(\mathbb{V}):=\overline{\mathbb{V}}{}^d$ (see also Figure \ref{fig:flowchart3}).

\begin{figure}[t]
\centering
\includegraphics[scale=0.8]{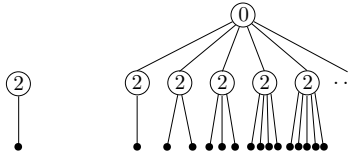}
\caption{\footnotesize (Left) An original vein $\mathbb{V}$; (Right) The closure $\mathbf{T}(\mathbb{V})$ of the vein $\mathbb{V}$.}
\label{fig:flowchart3}
\end{figure}


\begin{prop}\label{prop:closuredomination}
Every vein $\mathbb{V}$ is dominated by $\{\mathbf{T}(\mathbb{V})\}$, that is, every $\mathbb{V}$-piecewise continuous function is $\mathbf{T}(\mathbb{V})$-piecewise continuous.
\end{prop}

\begin{proof}
Given a branching function $b:\mathcal{V}^{\sf fin}\to\om$, consider the $\Pi^0_0$ cover $(P_{\sigma\fr i})$ for every $\sigma\in\overline{\mathcal{V}}{}^{\sf fin}$ defined by $P_{\sigma\fr i}=\mathcal{X}$ for $i=b(\iota(\sigma^-))$ and $P_{\sigma\fr i}=\emptyset$ for $i\not=b(\iota(\sigma^-))$.
By using these covers, it is not hard to check that $\mathbb{V}^b$-piecewise continuous functions are $\mathbf{T}(\mathbb{V})$-piecewise continuous.
\end{proof}

By a similar argument, given a countable collection $\mathbb{V}^\ast$ of veins, one can easily construct a vein $\overline{\mathbb{V}^\ast}$ and a labeled well-founded tree $\mathbf{T}(\mathbb{V}^\ast)$ on $\overline{\mathbb{V}^\ast}$ that dominates $\mathbb{V}^\ast$.
One can also introduce a representation of the set ${\rm Flow}(\mathbb{V})$ of all flows on $\Lambda$ via representations of branching functions $b$, $\mathbf{\Pi}^0_\xi$ sets $P_\sigma$ (via Borel codes), and continuous functions $f_\rho$.
Under such a representation, one can see that a function $f:\mathcal{X}\to\mathcal{Y}$ is $\mathbf{T}(\mathbb{V})$-piecewise continuous if there is a continuous function $\Lambda:\mathcal{X}\to{\rm Flow}(\mathbb{V})$ such that $f(x)=f_{\Lambda(x)}(x)$ for every $x\in\mathcal{X}$.

We say that $\mathbb{V}_0$ is {\em equivalent} to $\mathbb{V}_1$ if $\mathbb{V}_1$ dominates $\mathbb{V}_0$ and vice versa.
A vein $\mathbb{V}=(\mathcal{V},{\rm rk}_\mathcal{V})$ is {\em normal} if every non-terminal rank $0$ node is infinitely branching, and for every non-terminal node $\sigma\in \mathcal{V}$ of positive length, if the Borel rank of $\sigma$ is not greater than that of the immediate predecessor $\sigma^-$, then the number of immediate successors of $\sigma$ must be greater than that of $\sigma^-$, that is, $\mathcal{V}$ satisfies the following two conditions:
\begin{align*}
{\rm rk}_\mathcal{V}(\sigma)=0\;&\Longrightarrow\;{\rm br}_\mathcal{V}(\sigma)=\om,\\
{\rm rk}_\mathcal{V}(\sigma^-)\geq {\rm rk}_\mathcal{V}(\sigma)\;&\Longrightarrow\;1={\rm br}_\mathcal{V}(\sigma^-)<{\rm br}_\mathcal{V}(\sigma)=\om.
\end{align*}

\begin{lemma}\label{lem:normal}
Every vein is equivalent to a normal vein.
\end{lemma}

\begin{proof}
Suppose that $\sigma$ is a non-terminal finitely branching node such that ${\rm rk}_\mathcal{V}(\sigma)=0$.
Then, for a given labeled well-founded tree $\mathbf{V}=(V,{\rm rk}_V)$ on $\mathbb{V}$, a finite collection of rank $0$ sets $(U_i)_{i<k}$ will be placed on each node $\tau\in V$ with $\iota(\tau)=\sigma$.
Note that by the definition of a vein, the shapes below $\tau\fr i$ in $V$ for all $i<k$ are exactly the same.
Then, consider leaves in $V$ of the forms $\tau\fr i\fr \rho$ for $i<k$.
Since $(U_i)_{i<k}$ are of Borel rank $0$, by combining $(f_{\tau\fr i\fr \rho})_{i<k}$, one can easily get a single continuous function $f^\ast_{\tau\fr\rho}$.
Therefore, it causes no effect on $\mathbb{V}$-piecewise continuity even if we remove the node $\sigma$ from the vein $\mathbb{V}$.
For the latter condition of normality, if ${\rm rk}_\mathcal{V}(\sigma)\leq{\rm rk}_\mathcal{V}(\sigma^-)$ and ${\rm br}_\mathcal{V}(\sigma)\leq{\rm br}_\mathcal{V}(\sigma^-)$ then we can remove $\sigma$ from the vein as well.
\end{proof}

Consequently, we can always assume that, if our space is $2^\om$, every rank $0$ set assigned to a rank $0$ node is the clopen set generated by a single binary string $\eta\in 2^{<\om}$.
Hereafter we adopt this convention.
We also say that a vein $\mathbb{V}=(\mathcal{V},{\rm rk}_\mathcal{V})$ is {\em strongly normal} if it is normal, and moreover, for any non-terminal node $\sigma\in\mathcal{V}$ either the following condition (1) or (2) holds:
\begin{enumerate}
\item ${\rm rk}_\mathcal{V}(\sigma)<{\rm rk}_\mathcal{V}(\sigma^-)$ and ${\rm br}_\mathcal{V}(\sigma)>{\rm br}_\mathcal{V}(\sigma^-)$.
\item ${\rm rk}_\mathcal{V}(\sigma)>{\rm rk}_\mathcal{V}(\sigma^-)$ and ${\rm br}_\mathcal{V}(\sigma)<{\rm br}_\mathcal{V}(\sigma^-)$.
\end{enumerate}
Clearly, every strongly normal vein is normal.
To simplify our argument, in our main theorems, we assume strong normality of a vein; although the reader may find that a straightforward (but notationally complicated) modification of our proof gives us a similar result for non-strongly-normal veins.

We now introduce several operations on veins.
First we consider the finitary (infinitary) $\xi$-increment operation, which adds a new finitely (infinitely) $\mathbf{\Pi}^0_\xi$-branching node above the root $\langle\rangle$ of a given vein.
For a countable ordinal $\xi<\omega_1$ the {\em finite (infinite) $\xi$-increment} of a vein $\mathbb{V}=(\mathcal{V},{\rm rk}_\mathcal{V})$, denoted by $\mathbb{V}^{\oplus \xi}=(\mathcal{V}^{\oplus \xi},{\rm rk}_\mathcal{V}^{\oplus\xi})$ ($\mathbb{V}^{\oplus_\om \xi}=(\mathcal{V}^{\oplus_\om \xi},{\rm rk}_\mathcal{V}^{\oplus_\om\xi})$), is defined as follows:
\begin{align*}
\mathcal{V}^{\oplus\xi}&=\{\langle\rangle\}\cup\{\langle 0\rangle\fr\sigma:\sigma\in \mathcal{V}\},\quad
{\rm rk}_\mathcal{V}^{\oplus\xi}(\pair{})=\xi,\quad{\rm rk}_\mathcal{V}^{\oplus\xi}(0\fr\sigma)={\rm rk}_\mathcal{V}(\sigma),\\
\mathcal{V}^{\oplus_\om\xi}&=\{\langle\rangle\}\cup\{\langle n\rangle\fr\sigma:n\in\om\mbox{ and }\sigma\in \mathcal{V}\},\quad
{\rm rk}_\mathcal{V}^{\oplus_\om\xi}(\pair{})=\xi,\quad{\rm rk}_\mathcal{V}^{\oplus_\om\xi}(n\fr\sigma)={\rm rk}_\mathcal{V}(\sigma).
\end{align*}\medskip

\noindent Next we consider another operation.
Given a labeled well-founded tree $(T,{\rm rk}_T)$, we say that $\sigma\in T$ is {\em almost-terminal} if it is minimal among strings which have only finitely many successors in $T$, that is, there are only finitely many $\rho\in T$ extending $\sigma$, and every $\tau\prec\sigma$ has infinitely many successors in $T$.
More explicitly, for a leaf $\xi\in T^{\rm leaf}$, if the immediate predecessor $\xi^-$ is finitely branching, then we define ${\xi}^\ast=\xi^-$; otherwise, we define ${\xi}^\ast=\xi$.
If a vein $\mathbb{V}=(\mathcal{V},{\rm rk}_\mathcal{V})$ is normal, a string $\sigma\in \mathcal{V}$ is  almost-terminal if and only if it is of the form $\xi^\ast$ for some leaf $\xi\in \mathcal{V}^{\rm leaf}$.
Let $\mathcal{V}^{\rm at}$ denote the set of all almost-terminal nodes in $\mathcal{V}$.

The $\xi$-replacement operation converts each almost-terminal node (and all extensions) into an infinite $\mathbf{\Pi}^0_\xi$-branching node all of whose immediate successors are leaves.
For a countable ordinal $\xi<\omega_1$, the {\em $\xi$-replacement} of a vein $\mathbb{V}=(\mathcal{V},{\rm rk}_\mathcal{V})$, denoted by $\mathbb{V}^{\ominus\xi}=(\mathcal{V}^{\ominus\xi},{\rm rk}_\mathcal{V}^{\ominus\xi})$, is defined as follows (see also Figure \ref{fig:flowchart4}):
\begin{align*}
\mathcal{V}^{\ominus\xi}&=\{\sigma\in \mathcal{V}:\sigma\preceq\tau\mbox{ for some $\tau\in \mathcal{V}^{\rm at}$}\}\cup\{\tau\fr n:\tau\in \mathcal{V}^{\rm at}\mbox{ and }n\in\om\},\\
{\rm rk}_\mathcal{V}^{\ominus\xi}(\sigma)&=
\begin{cases}
\xi&\mbox{ if }\tau\in \mathcal{V}^{\rm at},\\
{\rm rk}_\mathcal{V}(\sigma)&\mbox{ otherwise}.
\end{cases}
\end{align*}

\begin{figure}[t]
\centering
\includegraphics[scale=0.8]{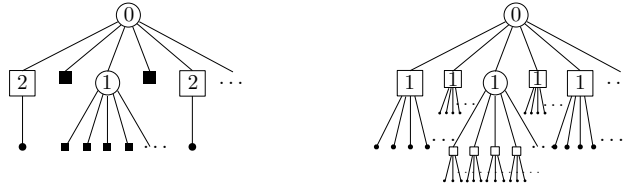}
\caption{\footnotesize (Left) An original vein $\mathbb{V}$, where almost-terminal nodes are surrounded by squares; (Right) The $1$-replacement $\mathbb{V}^{\ominus 1}$ of $\mathbb{V}$.}
\label{fig:flowchart4}
\end{figure}

Given a vein $\mathbb{V}$, we consider the following veins $\mathbb{V}'$ and $\mathbb{V}''$:
\begin{align*}
\mathbb{V}'&=
\begin{cases}
\mathbb{V}^{\ominus 1\oplus ({\rm rk}_\mathcal{V}(\langle\rangle)+1)}&\mbox{ if }{\rm br}_\mathcal{V}(\langle\rangle)=\om,\\
\mathbb{V}^{\ominus 1\oplus_\om 0\oplus 1}&\mbox{ if }{\rm br}_\mathcal{V}(\langle\rangle)=1,\\
\end{cases}\\
\mathbb{V}''&=
\begin{cases}
(\overline{\mathbb{V}^{\ominus 1}})^{\oplus ({\rm rk}_\mathcal{V}(\langle\rangle)+1)}&\mbox{ if }{\rm br}_\mathcal{V}(\langle\rangle)=\om,\\
(\overline{\mathbb{V}^{\ominus 1}})^{\oplus 1}&\mbox{ if }{\rm br}_\mathcal{V}(\langle\rangle)=1,\\
\end{cases}
\end{align*}

\begin{example}\label{example:double-prime}
For each $k\in\{1,\om\}$, let $\mathbb{V}_{\xi,k}$ be the vein such that $\mathcal{V}_{\xi,k}$ is a tree of height $2$ whose root is $k$-branching and labeled by $\xi$.
\begin{enumerate}
\item $(\mathbb{V}_{2,1})''$ is equivalent to $\mathbb{V}_{1,\om}$:
This is because $(\mathbb{V}_{2,1})^{\ominus 1}=\mathbb{V}_{1,\om}$, and $\overline{\mathbb{V}_{1,\om}}=\mathbb{V}_{1,\om}$.
Moreover, $(\mathbb{V}_{1,\om})^{\oplus 1}=\mathbb{V}_{1,1}\fr\mathbb{V}_{1,\om}$ is clearly equivalent to $\mathbb{V}_{1,\om}$.
\item $(\mathbb{V}_{1,\om}\fr\mathbb{V}_{2,1})''$ is equivalent to $\mathbb{V}_{2,1}\fr\mathbb{V}_{1,\om}$:
This is because, for $\mathbb{V}=\mathbb{V}_{1,\om}\fr\mathbb{V}_{2,1}$, $ \mathbb{V}^{\ominus 1}$ and hence $\overline{\mathbb{V}^{\ominus 1}}$ are equivalent to $\mathbb{V}_{1,\om}$, and $(\mathbb{V}_{1,\om})^{\oplus 2}=\mathbb{V}_{2,1}\fr\mathbb{V}_{1,\om}$, where note that ${\rm rk}_\mathcal{V}(\langle\rangle)+1=2$.
\item $(\mathbb{V}_{2,1}\fr\mathbb{V}_{1,\om}\fr\mathbb{V}_{2,1})''$ is equivalent to $\mathbb{V}_{1,1}\fr\mathbb{V}_{0,\om}\fr\mathbb{V}_{2,1}\fr\mathbb{V}_{1,\om}$:
This is because, for $\mathbb{V}=\mathbb{V}_{2,1}\fr\mathbb{V}_{1,\om}\fr\mathbb{V}_{2,1}$, $\mathbb{V}^{\ominus 1}=\mathbb{V}_{2,1}\fr\mathbb{V}_{1,\om}$ and therefore $\overline{\mathbb{V}^{\ominus 1}}=\mathbb{V}_{0,\om}\fr\mathbb{V}_{2,1}\fr\mathbb{V}_{1,\om}$.
\item Put $\mathbb{X}_m=\mathbb{V}_{1,\om}$ and $\mathbb{Y}_m=\mathbb{V}_{2,1}$ for any $m$.
In general, we have the following:
\begin{enumerate}
\item $(\mathbb{X}_0\fr\mathbb{Y}_0\fr\dots\fr\mathbb{X}_n\fr\mathbb{Y}_n)''$ is equivalent to $\mathbb{Y}_0\fr\mathbb{X}_0\fr\dots\fr\mathbb{Y}_n\fr\mathbb{X}_n$.
\item $(\mathbb{Y}_0\fr\mathbb{X}_0\fr\dots\fr\mathbb{Y}_n\fr\mathbb{X}_n\fr\mathbb{Y}_{n+1})''$ is equivalent to $\mathbb{V}_{1,1}\fr\mathbb{V}_{0,\om}\fr\mathbb{Y}_0\fr\mathbb{X}_0\fr\dots\fr\mathbb{Y}_n\fr\mathbb{X}_n$.
\end{enumerate}
\end{enumerate}
\end{example}

\section{Main Theorems}

\subsection{Topological Results}

We now consider coWadge/Weihrauch reducibility associated with classes of $\mathbb{V}$-piecewise continuous functions.
For a vein $\mathbb{V}$, we denote by $\mathbb{V}\mathcal{C}$ the class of all $\mathbb{V}$-piecewise continuous functions, and by $\sigma\mathcal{C}$ the class of all $\sigma$-continuous functions.
We often omit the symbol $\mathcal{C}$, e.g., we use the terminology such as $\mathbb{V}$-coWadge reducibility and $(\mathbb{V},\sigma)$-Weihrauch reducibility instead of $\mathbb{V}\mathcal{C}$-coWadge reducibility and $(\mathbb{V}\mathcal{C},\sigma\mathcal{C})$-Weihrauch reducibility.

We now focus on multi-valued functions which do not admit $\sigma$-continuous uniformizations.
However, this non-uniformizability property is not strong enough to obtain our main result, and so we will need to require functions to have a slightly stronger property.
For any known natural example $\mathcal{U}\subseteq 2^\om\times 2^\om$ which does not admit a $\sigma$-continuous uniformization, we may notice that even if we restrict the domain of $\mathcal{U}$ to any set $\mathcal{X}$ of {\em almost all} inputs, $\mathcal{U}\upr\mathcal{X}$ still does not admit a $\sigma$-continuous uniformization.
However, we also notice that if $\mathcal{U}$ is compact, then $\mathcal{U}$ admits a Borel (indeed, Baire-one) uniformization; therefore, for a set $\mathcal{X}$ of {\em almost all} inputs, $\mathcal{U}\upr\mathcal{X}$ admits a closed-piecewise continuous (i.e., layerwise continuous) uniformization by Luzin's theorem (indeed, if a $\sigma$-ideal $\mathcal{I}$ has the continuous reading of names, then $\mathcal{U}\upr\mathcal{X}$ admits a continuous uniformization on an $\mathcal{I}$-positive set $\mathcal{X}$).
The latter ``almost all'' is, of course, $\mu$-conullness with respect to the canonical product measure $\mu$ on $2^\om$ while the former ``almost all'' is $\mu$-conullness with respect to the Martin measure $\mu$ on $2^\om$.

A tree $E\subseteq 2^{<\om}$ is {\em pointed} if it is pruned (i.e., there is no leaf), and every infinite path through $E$ computes $E$ itself.
A perfect set $\mathcal{E}\subseteq 2^\om$ is {\em pointed} if it consists of all infinite paths through a pointed tree.
An important property of a pointed perfect set $\mathcal{E}$ is that $\mathcal{E}$ contains all Turing degrees above the degree of the base tree $E$.
For $A\subseteq 2^\om$ we put $\mu(A)=0$ if and only if $A$ has no pointed perfect subset, and this $\mu$ is called the {\em Martin measure} (under the axiom of determinacy).

As mentioned above, we do not know any {\em natural} example which does not admit a $\sigma$-continuous uniformization, but admits a $\sigma$-continuous uniformization on a pointed perfect set. (Of course, one can easily construct such an {\em artificial} example.)
Our requirement to $\mathcal{U}$ is to not admit a $\sigma$-continuous uniformization on a pointed perfect set.
In Section \ref{section:proof-main}, we will show the following:

\begin{theorem}\label{thm:main}
Let $\mathbb{V}$ be a strongly normal vein of Borel rank $(1,2)$.
For any nonempty compact sets $\mathcal{S},\mathcal{U}\subseteq 2^\om\times 2^\om$, if $\mathcal{U}$ does not admit a $\sigma$-continuous uniformization on a pointed perfect set, then there exists a compact set $\mathcal{T}\subseteq 2^\om\times 2^\om$ such that
\begin{enumerate}
\item $\mathcal{T}$ is coWadge reducible to $\mathcal{S}$.
\item $\mathcal{S}$ is $\mathbb{V}^{\prime\prime}$-coWadge reducible to $\mathcal{T}$.
\item $\mathcal{U}$ is not weakly $\mathbb{V}$-coWadge reducible to $\mathcal{T}$.
\end{enumerate}
\end{theorem}

Note that $\mathbb{V}$ itself may not be transitive; therefore, it seems better to consider the transitive closure ${\rm trcl}(\mathbb{V})$.
One can define the notion ${\rm trcl}(\mathbb{V})''$ in a straightforward manner.
Then, as a corollary of Theorem \ref{thm:main}, if $\mathbf{d}$ is a ${\rm trcl}(\mathbb{V})^{\prime\prime}$-coWadge degree containing a compact problem without $\sigma$-continuous uniformization on a pointed perfect set, then $\mathbf{d}$ contains an infinite decreasing chain of weak $\mathbb{V}$-coWadge degrees of compact problems.

We are also interested in whether our result gives a similar separation result for continuous Weihrauch reducibility.
Here we give a partial result.
We say that a function $h:2^\om\to 2^\om$ is {\em degree-invariant} if there is $c\in 2^\om$ such that, for any $x,y\geq_Tc$, the equation $x\equiv_Ty$ implies $h(x)\equiv_Th(y)$.
For a class $\mathcal{F}$ of functions, by $\mathcal{F}_{\rm inv}$ we mean that the class of all degree-invariant $\mathcal{F}$-functions.
It is clear that weakly $\mathbb{V}$-coWadge reducibility implies $(\mathbb{V},\mathcal{C}_{\rm inv})$-Weihrauch reducibility.
By symbols ${\sf ZF}$, ${\sf DC}$ and ${\sf AD}$ we denote the Zermelo-Fraenkel set theory (without choice), the axiom of dependent choice, and the axiom of determinacy, respectively.

\begin{theorem}[${\sf ZF}+{\sf DC}+{\sf AD}$]\label{thm:main2}
Let $\mathbb{V}$ be a strongly normal vein of Borel rank $(1,2)$.
For any nonempty compact sets $\mathcal{S},\mathcal{U}\subseteq 2^\om\times 2^\om$, if $\mathcal{U}$ does not admit a $\sigma$-continuous uniformization on a pointed perfect set, then there exists a compact set $\mathcal{T}\subseteq 2^\om\times 2^\om$ satisfying the assertions (1)--(4), where
\begin{enumerate}
\item[(4)] $\mathcal{U}$ is not $(\mathbb{V},\sigma_{\rm inv})$-Weihrauch reducible to $\mathcal{T}$.
\end{enumerate}
\end{theorem}

\subsection{Computable Results}\label{section:main-computable}

Now we consider the computable version of our main result.
For an oracle $z$, we say that a vein $\mathbb{V}=(\mathcal{V},{\rm rk}_\mathcal{V})$ of Borel rank $(1,2)$ is {\em $z$-computable} if the tree $\mathcal{V}\subseteq\om^{<\om}$ is $z$-computable, and the function ${\rm rk}_\mathcal{V}:\mathcal{V}\to\{0,1,2\}$ is computable.
A flow $\Lambda=(V,{\rm rk}_V,(P_\sigma)_{\sigma\in V},(f_\rho)_{\rho\in V^{\rm leaf}})$ on $\mathbb{V}$ is {\em $z$-computable} if the labeled well-founded tree $(V,{\rm rk}_V)$ is generated by a $z$-computable branching function, $P_\sigma$ is $\Pi^0_{{\rm rk}_V(\sigma)}(z)$ uniformly in $\sigma\in V$, and $f_\rho$ is partial $z$-computable uniformly in $\rho\in V^{\rm leaf}$, that is, there are $z$-computable functions $b:V^{\sf fin}\to\om$, $p:V\to\om$ and $\varphi:V^{\rm leaf}\to \om$ such that $(V,{\rm rk}_V)=(\mathcal{V}^b,{\rm rk}_{\mathcal{V}^b})$, $p(\sigma)$ is a $\Pi^0_{{\rm rk}_V(\sigma)}(z)$-index of $P_\sigma$, and $\varphi(\sigma)$ is an index of the partial $z$-computable function $f_\sigma$.

For a $z$-computable vein $\mathbb{V}$ we say that a function is {\em $\mathbb{V}$-piecewise $z$-computable} if it is of the form $f_\Lambda$ for some $z$-computable flow on $\mathbb{V}$.
We denote by $\mathbb{V}{\tt C}^z$ the class of $\mathbb{V}$-piecewise $z$-computable functions.
We also use the terminology such as $z$-computable $\mathbb{V}$-coWadge reducibility instead of $\mathbb{V}{\tt C}^z$-coWadge reducibility.

To state the computable version of our result, we need an effective version of $\sigma$-continuous non-uniformizability.
We first give a computability-theoretic interpretation of a $\sigma$-continuous uniformization.
One of the most fundamental results in Computability Theory is the equivalence between ``{\em (topological) continuity}'' and ``{\em oracle-computability}''.
This {\em relativization principle} for instance implies the following well-known fact in Computability Theory (see also Kihara \cite{Kih15}).

\begin{fact}[Folklore]
\hfill
\begin{enumerate}
\item A function $f:2^\om\to 2^\om$ is $\sigma$-continuous if and only if there is an oracle $z\in 2^\om$ such that $f(x)\leq_Tx\oplus z$ for any $x\in 2^\om$.
\item If a function $f:2^\om\to 2^\om$ is of Baire class $\alpha$ then there is an oracle $z\in 2^\om$ such that $f(x)\leq_T (x\oplus z)^{(\alpha)}$ for any $x\in 2^\om$.
\end{enumerate}
\end{fact}\medskip

\noindent As a corollary, we get the following characterization of a $\sigma$-continuous uniformization.
\begin{prop}
A multi-valued function $\mathcal{U}\subseteq 2^\om\times 2^\om$ admits a $\sigma$-continuous uniformization if and only if there is an oracle $z\in 2^\om$ such that $\mathcal{U}(x)$ has an $(x\oplus z)$-computable element for all $x\in{\rm dom}(\mathcal{U})$.
\end{prop}

We say that a function $f:2^\om\to 2^\om$ is {\em $\sigma$-computable} if there is a countable partition $(\mathcal{X}_n)_{n\in\om}$ of $2^\om$ such that $f\upr\mathcal{X}_n$ is computable for each $n\in\om$.
As in the above fact, one can easily see that $f$ is $\sigma$-computable if and only if $f(x)\leq_Tx$ for all $x$.
For this reason, $\sigma$-computability is traditionally called {\em non-uniform computability}.
We now observe the following property:

\begin{lemma}\label{lem:perfect-avoid}
For any compact set $\mathcal{U}\subseteq 2^\om\times 2^\omega$, if ${\rm dom}(\mathcal{U})$ is uncountable, and $\mathcal{U}$ does not admit a $\sigma$-continuous uniformization, then
\[N(\mathcal{U})=\{x\in 2^\om:\mathcal{U}(x)\mbox{ has no $x$-computable element}\}\]
has a perfect subset.
\end{lemma}

\begin{proof}
It is easy to see that $N(\mathcal{U})$ is Borel.
Therefore, if it has no perfect subset, then it has to be countable.
If it is countable, $\mathcal{U}$ has a $\sigma$-computable uniformization except for countably many points.
This implies that $\mathcal{U}$ has a $\sigma$-continuous uniformization. 
\end{proof}

A set $\mathcal{E}\subseteq 2^\om$ is {\em computably perfect} if there is a computable pruned perfect tree $E\subseteq 2^{<\om}$ such that $\mathcal{E}$ consists of all infinite paths through $E$.
Clearly, every computably perfect set is pointed.
We say that $\mathcal{U}$ is {\em computably non-$\sigma$-uniformizable} if $N(\mathcal{U})$ has a computably-perfect subset $\mathcal{E}$.
For instance, the problems in Example \ref{example:non-uniformizable} are computably non-$\sigma$-uniformizable via $\mathcal{E}=2^\om$.



\begin{theorem}\label{thm:main3}
Let $\mathbb{V}$ be a strongly normal computable vein of Borel rank $(1,2)$.
For any nonempty $\Pi^0_1$ sets $\mathcal{S},\mathcal{U}\subseteq 2^\om\times 2^\om$, if $\mathcal{U}$ is computably non-$\sigma$-uniformizable, then there exists a $\Pi^0_1$ set $\mathcal{T}\subseteq 2^\om\times 2^\om$ such that
\begin{enumerate}
\item $\mathcal{T}$ is computably coWadge reducible to $\mathcal{S}$.
\item $\mathcal{S}$ is computably $\mathbb{V}^{\prime\prime}$-coWadge reducible to $\mathcal{T}$.
\item $\mathcal{U}$ is not computably $(\mathbb{V},\sigma)$-Weihrauch reducible to $\mathcal{T}$.
\end{enumerate}
\end{theorem}

In this case, we do not require a $\sigma$-computable reduction to be degree invariant.
In particular, by effectivizing Proposition \ref{prop:weihrauch}, given such $\mathcal{S}$ and $\mathcal{U}$ we can effectively construct a compact-graph multifunction $\mathcal{T}$ on $2^\om$ such that
\begin{enumerate}
\item $\mathcal{S}$ is Weihrauch reducible to ${\rm TP}_{\Lambda'}\star\mathcal{T}$ for some flow $\Lambda'$ on $\mathbb{V}''$.
\item $\mathcal{U}$ is not Weihrauch reducible to ${\rm TP}_{\Lambda}\star\mathcal{T}$ for any flow $\Lambda$ on $\mathbb{V}$.
\end{enumerate}

In particular, Theorem \ref{thm:main3} implies the statement ($\dagger$) in Section \ref{sec:summary}.
We will give more details on how to verify the statement ($\dagger$) in the end of Section \ref{section:proof-main}.

\section{Proof of Main Theorems}\label{section:proof-main}

\subsection*{Convention}

Before starting the proof of our main theorems, without loss of generality, we may assume that
\[{\rm dom}(\mathcal{S})\mbox{ is uncountable, and }{\rm dom}(\mathcal{U})=2^\om.\]
This is because, if the domain of a uniformization problem is countable, then it is easy to see that the problem admits a $\sigma$-continuous uniformization.
Thus, ${\rm dom}(\mathcal{U})$ must be uncountable.
Moreover, it is easy to see that if a uniformization problem $\mathcal{U}$ is continuously $(\mathbb{V},\sigma)$-Weihrauch reducible to another problem $\mathcal{S}$, and if $\mathcal{S}$ admits a $\sigma$-continuous uniformization, then so does $\mathcal{U}$.
Therefore, if ${\rm dom}(\mathcal{S})$ is countable then $\mathcal{T}=\mathcal{S}$ satisfies the desired property.
It is known that each uncountable Polish space has a homeomorphic copy $\mathcal{C}$ of Cantor space $2^\om$.
Thus, we restrict the uniformization problem $\mathcal{U}$ to $\mathcal{C}$.
The difficulty of the uniformization problem $\mathcal{U}\upr\mathcal{C}$ is continuously equivalent to that of a uniformization problem whose domain is $2^\om$, and $\mathcal{U}\upr\mathcal{C}$ is reducible to $\mathcal{U}$.
Therefore, we can always assume that the domain of $\mathcal{U}$ is $2^\om$.
For an effective treatment, use the fact that every computably perfect computably presented Polish space has a computable homeomorphic copy of Cantor space.

\subsection*{Proof Strategy}

Suppose that $\mathcal{U}$ does not admit a $\sigma$-continuous uniformization on a pointed perfect set.
Then we first need the following lemma:

\begin{lemma}\label{lem:pointed}
If a compact-graph multifunction $\mathcal{U}$ on $2^\om$ does not admit a $\sigma$-continuous uniformization on a pointed perfect set, then there is a pointed perfect set $\mathcal{E}$ such that for every $x\in\mathcal{E}$, the value $\mathcal{U}(x)$ has no $x$-computable element.
\end{lemma}

\begin{proof}
By Martin's Cone Theorem (\cite{Martin68}; see also Marks-Slaman-Steel \cite[Lemma 3.5]{MarSlaSte}), $N(\mathcal{U})$ or its complement has a pointed perfect subset, where Borel determinacy is enough to show Martin's Cone Theorem for $N(\mathcal{U})$ because compactness of $\mathcal{U}$ implies its Borelness, and thus $N(\mathcal{U})$ is Borel.
Thus, $N(\mathcal{U})$ has a pointed perfect subset since its complement has no pointed perfect subset.
\end{proof}

Fix such $\mathcal{E}$, and let $d$ be a sufficiently powerful oracle making the pruned perfect tree $E$ be $d$-computable, and $\mathcal{T}$ be $\Pi^0_1(d)$.
The aim of this section is to show that the assertions (3) in Theorems \ref{thm:main} and \ref{thm:main3} and (4) in Theorem \ref{thm:main2} can be deduced from the following property:
\begin{itemize}[label=($\star$)]
\item For any $\mathbb{V}$-piecewise $r$-computable function $k$, 
\[(\forall x\in {\rm dom}(\mathcal{S}))(\forall z\in\mathcal{E})\;[r\leq_Tx\oplus d\equiv_Tz\;\rightarrow\;(\exists y\in\mathcal{T}(x))\;k(y)\not\in\mathcal{U}(z)].\]
\end{itemize}

\begin{lemma}\label{lem:property-star}
The property ($\star$) implies that:
\begin{enumerate}
\item $\mathcal{U}$ is not weakly $\mathbb{V}$-coWadge reducible to $\mathcal{T}$.
\item $\mathcal{U}$ is not computably $(\mathbb{V},\sigma)$-Weihrauch reducible to $\mathcal{T}$, whenever $d=\emptyset$ and $\mathcal{E}$ has a computable element.
\item $\mathcal{U}$ is not $(\mathbb{V},\sigma_{\rm inv})$-Weihrauch reducible to $\mathcal{T}$ under the axiom of determinacy.
\end{enumerate}
\end{lemma}

\begin{proof}
(1)
Suppose that $\mathcal{U}$ is weakly $\mathbb{V}$-coWadge reducible to $\mathcal{T}$.
Then there is a $\mathbb{V}$-piecewise $r$-computable function $k$ such that for any uniformization $t$ of $\mathcal{T}$, the composition $k\circ ({\rm id},t)$ uniformizes $\mathcal{U}$, that is, $k(x,t(x))\in\mathcal{U}(x)$ for all $x\in 2^\om$.
Choose $x\in\mathcal{E}$ with $r\leq_Tx$.
Such $x$ exists since $\mathcal{E}$ is pointed.
The function $k_0$ defined by $k_0(y)=k(x,y)$ for any $y$ is also $\mathbb{V}$-piecewise $x$-computable.
Therefore, by the property ($\star$), there is a uniformization $t_0$ of $\mathcal{T}$ such that $k_0\circ t_0(x)\not\in\mathcal{U}(x)$.
However this contradicts our assumption since $k(x,t_0(x))=k_0\circ t_0(x)$.

(2)
Next suppose for the sake of the contradiction that $\mathcal{U}$ is computably $(\mathbb{V},\sigma)$-Weihrauch reducible to $\mathcal{T}$.
Then there are a $\mathbb{V}$-piecewise computable function $k$ and a $\sigma$-computable function $h$ such that for any uniformization $t$ of $\mathcal{T}$, the function $k\circ({\rm id},t\circ h)$ uniformizes $\mathcal{U}$, that is, $k(z,t(h(z)))\in\mathcal{U}(z)$ for all $z\in 2^\om$.
Let $z\in\mathcal{E}$ be a computable element.
Then, $h(z)$ is also computable since $\sigma$-computability of $h$ implies that $h(z)\leq_Tz$.
In particular, $z\equiv_Th(z)$.
Moreover, the function $k_0$ defined by $k_0(y)=k(z,y)$ for any $y$ is also $\mathbb{V}$-piecewise computable.
Therefore, by the property ($\star$), there is a uniformization $t_0$ of $\mathcal{T}$ such that $k_0\circ t_0(h(z))\not\in\mathcal{U}(z)$.
However this contradicts our assumption since $k(z,t_0(h(z)))=k_0\circ t_0(h(z))$.

(3)
Suppose that $\mathcal{U}$ is $(\mathbb{V},\sigma_{\rm inv})$-Weihrauch reducible to $\mathcal{T}$, that is, for a sufficiently powerful oracle $r$, there are a $\mathbb{V}$-piecewise $r$-computable function $k$ and a degree-invariant $\sigma$-computable-relative-to-$r$ function $h$ such that for any uniformization $t$ of $\mathcal{T}$, the function $k\circ ({\rm id},t\circ h)$ uniformizes $\mathcal{U}$, that is, $k(z,t(h(z)))\in\mathcal{U}(z)$ for all $z\in 2^\om$.
By $\sigma$-computability of $h$ relative to $r$, we always have $h(z)\leq_Tz$ for all $z\geq_Tr$.
As before, if $z\in\mathcal{E}$ then we cannot have $r\leq_Th(z)\oplus d\equiv_Tz$; therefore, it must hold that $h(z)<_Tz$ for all $z\geq_Tr\oplus d$ and $z\in\mathcal{E}$.
In particular, $h(z)<_Tz$ on a cone, that is, there is $c\in 2^\om$ such that $h(z)<_Tz$ for all $z\geq_Tc$ since $h$ is degree-invariant and $\mathcal{E}$ is pointed (that is, $\mathcal{E}$ contains a Turing cone).
By degree-invariance of $h$, we can use the Slaman-Steel Theorem \cite[Theorem 2]{SlaSte88} to get that $h$ is constant on a cone, that is, there are $c,y\in 2^\om$ such that $h(z)\equiv_Ty$ for all $z\geq_Tc$.
Now recall that every compact set admits a Baire-one uniformization, so let $t$ be such a Baire-one uniformization of $\mathcal{T}$.
In particular, whenever $z\geq_Tc$, we have $t(h(z))\leq_T(y\oplus u)'$ for some oracle $u\in 2^\om$.
Therefore, $k(z,t(h(z)))\leq_Tr\oplus z\oplus (y\oplus u)'$ holds for all $z\geq_Tc$.
Note that $v:=c\oplus r\oplus (y\oplus u)'$ is a constant, and we have $k(z,t(h(z)))\leq_Tz\oplus v$ for all $z\in 2^\om$.
This would imply that $z\mapsto k(z,t(h(z)))$ is $\sigma$-continuous.
However, since $k(z,t(h(z)))\in\mathcal{U}(z)$, this would give a $\sigma$-continuous uniformization of $\mathcal{U}$, which is a contradiction.
Note that if we only consider degree-invariant $\sigma$-continuous {\em Borel} functions, then we can avoid the use of the axiom of determinacy.
\end{proof}

We do not know whether the property ($\star$) implies the similar separation result for $(\mathbb{V},\sigma)$-Weihrauch reducibility as well.

\subsection*{Approximation of Trees}

To prove the main theorems, we will need the property ($\star$).
We first fix a sufficiently powerful oracle $d$ which ensures that both $\mathcal{S}$ and $\mathcal{U}$ are $\Pi^0_1(d)$.
Let $\mathcal{E}$ be a pointed perfect set as in Lemma \ref{lem:pointed}, so that $\mathcal{U}(z)$ has no $z$-computable element for any $z\in\mathcal{E}$.
Without loss of generality, we may assume that $d\equiv_TE$, where $E$ is a pointed tree generating $\mathcal{E}$.
This is because, for a canonical $E$-computable homeomorphism $\psi:2^\om\to\mathcal{E}$, we consider $\mathcal{E}_d=\{\psi(x\oplus d):x\in 2^\om\}$.
It is easy to see that $\mathcal{E}_d$ is $\Pi^0_1(d\oplus E)$, and that $z\geq_Td\oplus E$ holds for all $z\in\mathcal{E}_d$.
Then, we replace $d$ and $\mathcal{E}$ with $d\oplus E$ and $\mathcal{E}_d$, respectively.

Our $d$-computable construction of a compact set $\mathcal{T}$ will be {\em fiber-wise}, that is, we will construct an $(x\oplus d)$-computable tree $T(x)$ uniformly in $x$ (where $\mathcal{T}(x)$ is the $x$-th fiber of the projection of $\mathcal{T}$ into the first coordinate).
Hereafter, by $S(x)$, $T(x)$ and $U(x)$ we denote the $(x\oplus d)$-computable trees whose infinite paths form the fibers $\mathcal{S}(x)$, $\mathcal{T}(x)$ and $\mathcal{U}(x)$, respectively.
Such trees exist since $\mathcal{S}(x)$, $\mathcal{T}(x)$ and $\mathcal{U}(x)$ are $\Pi^0_1(x\oplus d)$ subset of $2^\om$ uniformly in $x$.
On each fiber $T(x)$ our strategy looks at the fibers $(U(z):x\oplus d\equiv_Tz\in\mathcal{E})$.
If $x\oplus d\equiv_Tz\in\mathcal{E}$ then, in particular, $z\geq_Td$, and therefore, by the property of $d$ mentioned above, $\mathcal{U}(z)$ has no $(x\oplus d)$-computable element.

Now we describe a uniform $(x\oplus d)$-computable approximation of the collection $(U(z):x\oplus d\equiv_Tz\in\mathcal{E})$.
Let $\Phi_i^d$ be the $i$-th partial $d$-computable function, and $\Phi_j$ be the $j$-th partial computable function.
If $z\equiv_Tx\oplus d$ then there are indices $i$ and $j$ such that $z=\Phi^d_i(x)$ and $x\oplus d=\Phi_j(z)$.
In particular, $\Phi_j\circ\Phi_i^d(x)=x\oplus d$.
We will define a tree $U^{x}_{i,j}$ for each pair $(i,j)$ of indices such that if $\Phi_j\circ\Phi_i^d(x)=x\oplus d$ and $\Phi_i^d(x)\in\mathcal{E}$ hold then $U^{x}_{i,j}=U(\Phi_i^d(x))$; otherwise $U^{x}_{i,j}$ is a finite tree.
This ensures that $U^x_{i,j}$ has no $(x\oplus d)$-computable infinite path for any $i,j\in\om$.

We first note that since $\mathcal{U}$ is $\Pi^0_1(d)$, there is a $d$-computable map sending each $z$ into a $\Pi^0_1(z\oplus d)$-code of the fiber $\mathcal{U}(z)$.
In other words, it is straightforward to see that there is a uniformly $d$-computable way of approximating all fibers of $\mathcal{U}$ as follows:
\begin{itemize}
\item Given a string $\tau\in 2^{<\om}$, $U(\tau)$ is a finite tree of height $|\tau|$.
\item If $\sigma\prec\tau$ then $U(\tau)\setminus U(\sigma)$ consists only of strings of length greater than $|\sigma|$.
\item $U(z)=\bigcup_nU(z\upr n)$ for all $z\in 2^\om$.
\end{itemize}

Given $\sigma\in 2^{<\om}$, as usual, by $\Phi_i(\sigma)$ we denote a binary string obtained by the stage $|\sigma|$-approximation of the $i$-th Turing machine computation $\Phi_i$ by using $\sigma$ as an oracle.
Given $s$, let $\ell_{i,j}[s]$ be the maximal length $\ell\in\om$ such that
\[\Phi_j\circ\Phi_i^d(x\upr s)\upr \ell=(x\oplus d)\upr \ell\mbox{, and }\Phi_i^d(x\upr s)\upr \ell\in E.\]

Then we define the stage $s$-approximation of $U^x_{i,j}$ as follows:
\[U^x_{i,j}[s]=\{\sigma\in U(\Phi_i^d(x\upr s)):|\sigma|<l_{i,j}[s]\}.\]
It is not hard to see that $U^x_{i,j}:=\bigcup_{s\in\om}U^x_{i,j}[s]$ satisfies the desired condition, that is, if $z\in\mathcal{E}$ and $z\equiv_Tx\oplus d$ via indices $(i,j)$ then $U^x_{i,j}=U(z)$, and if $(i,j)$ is not a correct pair of indices then $U^x_{i,j}$ is finite.

\subsection*{Enumeration of Flows}

To show our main theorems, we need the notion of a partial flow.
A {\rm partial flow} on a vein $\mathbb{V}$ is a pair $\Lambda=(\mathbf{V},\Gamma)$ of a labeled well-founded tree $\mathbf{V}=(V,{\rm rk}_V)$ and a {\em partial} flowchart $\Gamma$ on $\mathbf{V}$ such that $\mathbf{V}$ is a labeled subtree of $\mathbb{V}^b$ for some branching function $b$.
Here, a partial flowchart $\Gamma$ is a tuple $(P_\xi,f_\xi)_{\xi\in V}$ such that $P_\xi$ is a $\mathbf{\Pi}^0_{{\rm rk}_V(\xi)}$ set and $f_\xi$ is a partial continuous function.
As in Definition \ref{def:labeledwellfounded}, for a given $x$, the leftmost leaf $\rho\in V^{\rm leaf}$ such that $x\in\bigcap_{\sigma\preceq\rho}P_\sigma$ is said to be the true path of $\Lambda$ along $x$, and written as ${\rm TP}_\Lambda(x)$, {\em if such $\rho$ exists}.
Here, note that we do {\em not} require that $(P_{\xi\fr n})_{n}$ be a cover of $P_\xi$, and the domain of $f_\xi$ include $D_\xi$, where $D_\xi$ is the set of all $x$ such that ${\rm TP}_\Lambda(x)$ is defined, and ${\rm TP}_\Lambda(x)=\xi$.
A partial flow $\Lambda$ always defines a {\em partial} function $f_\Lambda$ by $f_\Lambda(x)=f_{{\rm TP}_\Lambda(x)}(x)$ for any $x$ such that ${\rm TP}_\Lambda(x)$ defines some value $\xi$ and $x\in {\rm dom}(f_\xi)$.
The set of all such $x$'s is called the actual domain of $\Lambda$, and denoted by ${\rm dom}(\Lambda)$, which is now possibly different from $P_{\langle\rangle}$.

Hereafter we assume that all veins $\mathbb{V}$ are of Borel rank $(1,2)$.
Let $\mathbb{V}=(\mathcal{V},{\rm rk}_\mathcal{V})$ be a vein, and $\Lambda=(V,{\rm rk}_V,(P_\xi),(f_\xi))$ be a partial flow on the vein $\mathbb{V}$, where $(V,{\rm rk}_V)$ is of the form $(\mathcal{V}^b,{\rm rk}_{\mathcal{V}^b})$ for some branching function $b$.
Then, one can define a $\Pi^0_2$ ($\Pi^0_1$, resp.)~formula $p$ ($q$, resp.) on $V\times\om\times 2^{<\om}\times 2^{<\om}$ (with a parameter $z$), a partial function $\eta:\subseteq V\times\om\to 2^{<\om}$, and a partial continuous function $f:\subseteq V^{\rm leaf}\times 2^\om\to 2^\om$ as follows:
\begin{align*}
x\in P_{\xi\fr n}\;&\iff\;
\begin{cases}
(\forall s)(\exists t\geq s)\;p(\xi,n,x\upr t,z\upr t)&\mbox{ if ${\rm rk}_V(\xi)=2$},\\
(\forall s)\;q(\xi,n,x\upr s,z\upr s)&\mbox{ if ${\rm rk}_V(\xi)=1$},\\
x\upr |\eta(\xi,n)|=\eta(\xi,n)&\mbox{ if ${\rm rk}_V(\xi)=0$},
\end{cases}\\
f_\xi(x)&=f(\xi,x)\qquad\mbox{if $\xi\in V^{\rm leaf}$ and $x\in{\rm dom}(f_\xi)$.}
\end{align*}
Here recall our convention mentioned after Lemma \ref{lem:normal} that a rank $0$ set assigned to a rank $0$ node is generated by a single binary string. 
If we know all information on $(b,p,q,\eta,f)$ and $z$, then we can recover the partial flow $\Lambda$.
We now introduce an enumeration of (partial) flows on a fixed vein $\mathbb{V}=(\mathcal{V},{\rm rk}_\mathcal{V})$.

Given an oracle $z\in 2^\om$ and $e=\pair{e_0,e_1,e_2,e_3,e_4}$, consider the tuple $(\varphi^z_{e_0},p_{e_1},q_{e_2},\eta^z_{e_3},\Phi^z_{e_4})$, where $\varphi^z_{d}:\subseteq\mathcal{V}^{\sf fin}\to\om$ is the $d$-th partial $z$-computable function, $p_d$ ($q_d$, resp.)~is the $d$-th $\Pi^0_2$ ($\Pi^0_1$, resp.)~formula on $\om^{<\om}\times\om\times 2^{<\om}\times 2^{<\om}$, $\eta_d^z:\om^{<\om}\times\om\to 2^{<\om}$ is the $d$-th partial $z$-computable function, and $\Phi_d:\subseteq \om^{<\om}\times 2^\om\to 2^\om$ is the $d$-th partial computable function.

Note that our branching function $\varphi^z_{e_0}$ is partial, so we need to describe how to produce a labeled well-founded tree from a partial branching function.
The idea is that we put successors of $\sigma$ into our tree only after $\varphi^z_{e_0}(\sigma)$ returns some outcome.
In other words, if the computation $\varphi^z_{e_0}(\sigma)$ never halts, all successors of $\sigma$ would vanish.
We also note that $\eta^z_{e_3}$ is partial as well, so we also require $\eta^z_{e_3}(\sigma,n)$ to be defined before putting the node $\sigma\fr n$ into our tree. 
Based on the above idea, we define $\mathbf{V}^z_{e,s}=(V^z_{e,s},{\rm rk}^z_{e,s})$, {\em the stage $s$ approximation of the $e$-th $z$-computable branching of $\mathbb{V}$} (with a {\em copy-source-referring} function $\iota:V^z_{e,s}\to \mathcal{V}$), as follows:
\begin{enumerate}
\item $\pair{}\in V^z_{e,s}$ and $\iota(\pair{})=\pair{}$.
\item If $\sigma\in V^z_{e,s}$, ${\rm br}_\mathcal{V}(\iota(\sigma))=1$ and the computation $\varphi^z_{e_0}(\iota(\sigma))$ converges by stage $s$, then $\sigma$ is converted into a $\varphi^z_{e_0}(\iota(\sigma))$-branching node, that is,
\begin{align*}
\pair{{\rm rk}^z_{e,s}(\sigma),{\rm br}_{V^z_{e,s}}(\sigma)}=&\pair{{\rm rk}_\mathcal{V}(\iota(\sigma)),\varphi^z_{e_0}(\iota(\sigma))},\\
\sigma\fr n\in V^z_{e,s}, \quad\quad & \iota(\sigma\fr n)=\iota(\sigma)\fr\ast,
\end{align*}
for every $n<\varphi^z_{e_0}(\sigma)$.
Here $\iota(\sigma)\fr\ast$ is the unique immediate successor of $\iota(\sigma)$ in $\mathcal{V}$.
\item If $\sigma\in V^z_{e,s}$ and ${\rm br}_\mathcal{V}(\iota(\sigma))=\om$, then we proceed as follows:
First, if the Borel rank of $\sigma$ is greater than $0$, then $\sigma$ remains the same as the copy source node $\iota(\sigma)$.
If the Borel rank of $\sigma$ is $0$ and moreover the associated rank $0$ set $\eta^z_e(\sigma,n)$ is determined at stage $s$, then $\sigma$ remains the same as the copy source node $\iota(\sigma)$, as well.
Otherwise (that is, the $e_3$-th computation has not yet computed the associated rank $0$ set $\eta^z_{e_3}(\sigma,n)$ at stage $s$) we do not put $\sigma\fr n$.
Formally speaking, we always define
\[\pair{{\rm rk}^z_{e,s}(\sigma),{\rm br}_{V^z_{e,s}}(\sigma)}=\pair{{\rm rk}_\mathcal{V}(\iota(\sigma)),\om},\]
and moreover, for any $n$ such that $\iota(\sigma)\fr n\in \mathcal{V}$, 
\[\sigma\fr n\in V^z_{e,s}, \quad\quad  \iota(\sigma\fr n)=\iota(\sigma)\fr n,\]
 if ${\rm rk}_\mathcal{V}(\sigma)>0$; or if ${\rm rk}_\mathcal{V}(\sigma)=0$ and the computation $\eta^z_{e_3}(\sigma,n)$ converges by stage $s$.
\item If $\sigma\in V^z_{e,s}$ and ${\rm br}_\mathcal{V}(\iota(\sigma))=0$, then $\sigma$ remains the same as the copy source node $\iota(\sigma)$, that is,
\begin{align*}
\pair{{\rm rk}^z_{e,s}(\sigma),{\rm br}_{V^z_{e,s}}(\sigma)}=&\pair{{\rm rk}_\mathcal{V}(\iota(\sigma)),0}.
\end{align*}
Moreover, we put $\sigma$ into $V^{z,{\rm leaf}}_{e,s}$.
Note that even if $\sigma$ is a leaf of $V^z_{e,s}$, the leaf $\sigma$ may not belong to $V^{z,{\rm leaf}}_{e,s}$.
\end{enumerate}\medskip

\noindent We call $\Lambda^z_{e,s}=(V^z_{e,s},{\rm rk}^z_{e,s},p_{e_1},q_{e_2},\eta^z_{e_3},f^z_{e_4})$ the {\em stage $s$ approximation of the $e$-th partial $z$-computable flow} on $\mathbb{V}$.
We then recover $(P^{z,e}_\xi:\xi\in V^z_{e,s})$ by using the above mentioned equivalence.

\subsection*{Weak-Totalization of Flows}

Without loss of generality, we can always assume that for any $\xi\in V^z_{e,s}$, if $\xi$ is finitely branching and $\varphi^z_{e_0}(\iota(\xi))$ is defined by stage $s$ then $(P^{z,e}_{\xi\fr n})$ covers $P^{z,e}_\xi$ (by assuming that the rightmost immediate successor of a finite branching node accepts all reals $x$).
However, it is not generally true for infinite branching nodes.
To get the covering property for infinite branching nodes (by modifying our tree $V^z_{e,s}$), we note that, since our vein is strongly normal, the predecessor of an infinite branching node $\zeta$ of positive length is finitely branching and ${\rm rk}^z_{e,s}(\zeta)<{\rm rk}^z_{e,s}(\zeta^-)$.
Now consider a finite branching node $\xi\in V^z_{e,s}$ with successors $(\xi\fr n)_{n<c}$.
We double the number of branches of $\xi$, and consider:
\[P^{z,e}_{\xi\fr 2n}=2^\om\setminus\bigcup_mP^{z,e}_{\xi\fr n\fr m},\; P^{z,e}_{\xi\fr 2n+1}=P^{z,e}_{\xi\fr n},\mbox{ and }P^{z,e}_{\xi\fr 2n+1\fr m}=P^{z,e}_{\xi\fr n\fr m}.\]
Note that the Borel complexity of $P^{z,e}_{\xi\fr 2n}$ is ${\rm rk}^z_e(\xi\fr n)+1\leq{\rm rk}^z_e(\xi)$ by normality, and $P^{z,e}_{\xi\fr n}$ is covered by $(P^{z,e}_{\xi\fr n\fr m})_m$.
It is clear that this modification does not produce any change on the generated function $f_{\Lambda^z_e}$.
We call this procedure the {\em weak-totalization} of a given vein.
We will give a formal description of weak-totalization below.

A partial flow $\Lambda=(V,{\rm rk}_V,p,q,\eta,f)$ automatically yields $(P_\xi)_{\xi\in V}$ where $P_{\pair{}}=2^\om$ (note that $P_{\pair{}}$ is possibly different from the actual domain of the generated function $f_\Lambda$).
Such a flow is said to be {\em weakly total} if for every non-terminal $\xi\in V$, $P_\xi$ is covered by $(P_{\xi\fr n})_{n}$.
Let $\Lambda=(V,{\rm rk}_V,p,q,\eta,f)$ be a partial flow on a vein $\mathbb{V}$.
We first note that for any $\sigma\in V$ we have that ${\rm rk}_V(\sigma\fr i)={\rm rk}_V(\sigma\fr j)$ and that ${\rm br}_V(\sigma\fr i)=\om$ if and only if ${\rm br}_V(\sigma\fr j)=\om$ whenever $\sigma\fr i,\sigma\fr j\in V$ since $V$ is obtained as the $b$-branching of a vein for some branching function $b$.
Now we focus on a non-terminal string $\sigma\in V$ such that ${\rm br}_V(\sigma)<\om$ and ${\rm br}_V(\sigma\fr j)=\om$ for some/any $j$.
Let $V^{\ast}$ be the set of all such strings.
We define the {\em weak-totalization} $\Lambda^{\rm tot}=(V^{\rm tot},{\rm rk}_V^{\rm tot},p^{\rm tot},q^{\rm tot},\eta^{\rm tot},f^{\rm tot})$ (with a copy-source-referring function $\iota:V^{\rm tot}\to V$) as follows:
\begin{enumerate}
\item 
If the root of $V$ is infinitely branching, i.e., ${\rm br}_V(\pair{})=\om$, then we add a new two-branching node above the root:
\begin{align*}
\pair{{\rm rk}_V^{\rm tot}(\pair{}),{\rm br}_{V^{\rm tot}}(\pair{})}&=\pair{{\rm rk}_V(\pair{})+1,2},\\
\pair{},\pair{0},\pair{1}\in V^{\rm tot},\;& \iota(\pair{0})=\star,\;\iota(\pair{1})=\pair{}.
\end{align*}
Here $\star$ is a fixed new symbol which is not contained in $V$ (see the items (4)--(5)).
\item If $\sigma\in V^{{\rm tot}}$, $\iota(\sigma)\in V^\ast$, then we double the number of branches of $\sigma$, that is,
\begin{align*}
\pair{{\rm rk}_V^{\rm tot}(\sigma),{\rm br}_{V^{\rm tot}}(\sigma)}=&\pair{{\rm rk}_V(\iota(\sigma)),2\cdot{\rm br}_V(\iota(\sigma))},\\
\sigma\fr 2n\in V^{\rm tot}, \quad\quad & \iota(\sigma\fr 2n)=\star,\\
\sigma\fr 2n+1\in V^{\rm tot}, \qquad & \iota(\sigma\fr 2n+1)=\iota(\sigma)\fr n
\end{align*}
for every $n<{\rm br}_V(\iota(\sigma))$ such that $\iota(\sigma)\fr n\in V$.
\item If $\sigma\in V^{\rm tot}$ and $\iota(\sigma)\not\in V^\ast$, then $\sigma$ remains unchanged, that is,
\begin{align*}
\pair{{\rm rk}_V^{\rm tot}(\sigma),{\rm br}_{V^{\rm tot}}(\sigma)}&=\pair{{\rm rk}_V(\iota(\sigma)),{\rm br}_V(\iota(\sigma))},\\
\sigma\fr n\in V^{\rm tot}, \quad\quad & \iota(\sigma\fr n)=\iota(\sigma)\fr n,
\end{align*}
for every $n\in\om$ such that $\iota(\sigma)\fr n\in V$.
Moreover, if $\iota(\sigma)\in V^{{\rm leaf}}$, then we declare that $\sigma\in V^{{\rm tot},{\rm leaf}}$.
\item If $\sigma\in V^{\rm tot}$ and $\iota(\sigma)=\star$, then we declare that $\sigma$ is a leaf, that is,
\begin{align*}
\pair{{\rm rk}_V^{\rm tot}(\sigma),{\rm br}_{V^{\rm tot}}(\sigma)}=&\pair{0,0}.
\end{align*}
and we declare that $\sigma\in V^{{\rm tot},{\rm leaf}}$.
\item For every $\sigma\in V^{\rm tot}$, if $\iota(\sigma)\in V^\ast$, say ${\rm rk}_V(\iota(\sigma))=2$ and ${\rm rk}_V(\iota(\sigma)\fr j)=1$ for some/any $j$, then we define 
\begin{align*}
p^{\rm tot}(\sigma,2n,\alpha,\beta)\;&\iff\;(\forall m)\;\neg q(\iota(\sigma)\fr n,m,\alpha,\beta),\\
p^{\rm tot}(\sigma,2n+1,\alpha,\beta)\;&\iff\;p(\iota(\sigma),n,\alpha,\beta).
\end{align*}
For other cases, we also define the corresponding formulas on $\sigma$ by the same way.
if $\iota(\sigma)\not\in V^\ast$, then we define $p^{\rm tot}$ as follows:
\[p^{\rm tot}(\sigma,n,\alpha,\beta)\;\iff\;p(\iota(\sigma),n,\alpha,\beta)\]
Similarly, we also define $q^{\rm tot}$ and $\eta^{\rm tot}$ in the same way.
For $\sigma\in V^{{\rm tot},{\rm leaf}}$ with $\iota(\sigma)\not=\star$, we define $f^{\rm tot}(\sigma,\cdot)=f(\iota(\sigma),\cdot)$ as well.
Finally, if $\iota(\sigma)=\star$, then we define $f^{\rm tot}(\sigma,\cdot)$ as a nowhere defined function.
\end{enumerate}\medskip

\noindent It is not hard to check that $\Lambda^{\rm tot}$ is weakly total and $f_\Lambda=f_{\Lambda^{\rm tot}}$.
By strong-normality mentioned above, it is also clear that the underlying vein $\mathbb{V}^{\rm tot}$ of the weak-totalization $\Lambda^{\rm tot}$ is of the form $\mathbb{V}^{\oplus{\rm rk}_V(\langle\rangle)+1}$ if the root of $V$ is infinitely branching; otherwise $\mathbb{V}^{\rm tot}=\mathbb{V}$.
We now think of each $e\in\om$ as an index of weakly total flows $(\Lambda^{z,{\rm tot}}_{e,s})_{s\in\om}$.


\subsection*{Requirements}

To prove Theorems \ref{thm:main}, \ref{thm:main2} and \ref{thm:main3}, we will construct a $(z\oplus d)$-computable tree $T(z)\subseteq 2^{<\om}$ uniformly in $z$ which fulfills the following requirements:
\begin{align*}
\mathcal{G}:&\;(\exists g\in\mathbb{V}^{\prime\prime}\mathtt{C}^d)(\forall x\in\mathcal{T}(z))\;g(z,x)\in\mathcal{S}(z),\\
\mathcal{N}^z_{e,i,j}:&\;\mathcal{S}(z)\not=\emptyset\;\longrightarrow\;(\exists x\in\mathcal{T}(z))\;f_{\Lambda^{z\oplus d}_e}(x)\not\in\mathcal{U}^z_{i,j}.
\end{align*}
where recall that $\mathbb{V}^{\prime\prime}\mathtt{C}^d$ is the class of all $\mathbb{V}^{\prime\prime}$-piecewise $d$-computable functions.

The global requirement $\mathcal{G}$ clearly ensures the assertion (2) in Theorems \ref{thm:main} and \ref{thm:main3}.
The requirements $(\mathcal{N}^z_{e,i,j})$ ensure the assertions (3) in Theorems \ref{thm:main} and \ref{thm:main3} and (4) in Theorem \ref{thm:main2}.
To see this, it suffices to check that the requirements $(\mathcal{N}^z_{e,i,j})$ entail the property ($\star$) since the property ($\star$) implies 
these assertions as mentioned before.
Let $k$ be a $\mathbb{V}$-piecewise $r$-computable function for some $r\leq_Tz\oplus d$.
Then there is an index $e$ such that $k=\Lambda^{z\oplus d}_e$.
Moreover, if $y\equiv_Tz\oplus d$ then there are indices $i,j$ such that $\mathcal{U}(y)=\mathcal{U}^z_{i,j}$.
Therefore, by choosing a uniformization $t$ of $\mathcal{T}$ satisfying $t(z)=x$ for an $x$ in the above $\mathcal{N}^z_{e,i,j}$, we have $k\circ t(z)\not\in\mathcal{U}(y)$ as desired.

To simplify our argument, we assume that $z=d=\emptyset$.
The proof for general $z$ and $d$ is a straightforward relativization of our strategy for $z=d=\emptyset$.
Moreover, for instance, we use the symbols $S,T,U$ instead of $S(\emptyset),T(\emptyset),U(\emptyset)$, respectively, if there is no confusion.

First note that if $\mathcal{S}:=\mathcal{S}(\emptyset)$ is nonempty, there are infinitely many strings $\rho$ such that $\rho$ is a minimal string which is not contained in the tree $S$. 
Let $\rho_e$ be the $e$-th such string.
Clearly $\rho_d$ is incomparable with $\rho_e$ whenever $d\not=e$.
For any $e,i,j,s\in\om$, we will construct a computable monomorphism $\gamma^{e,i,j}_s:S\to 2^{<\om}$, that is, $\sigma\preceq\tau$ if and only if $\gamma^{e,i,j}_s(\sigma)\preceq\gamma^{e,i,j}_s(\tau)$.
The monomorphism $\gamma^{e,i,j}_s$ also satisfies that $\gamma^{e,i,j}_s(\alpha)\succeq\rho_{\langle e,i,j\rangle}$ for any $\alpha\in S$.
Then the stage $s$ approximation of our tree $T_s$ will be defined as follows:
\[
T_s=S\cup\{\tau\in 2^{<\om}:(\exists e,i,j)(\exists\alpha\in S)\;\tau\preceq \gamma^{e,i,j}_s(\alpha)\}.
\]
Moreover, we will ensure that $\gamma^{e,i,j}_{s+1}(\alpha)\in T_s$ for all $\alpha\in S$ and $s\in\om$.
Therefore $T$ will be defined as follows:
\[
T=\bigcap_{s\in\om}T_s=S\cup\{\tau\in 2^{<\om}:(\exists e,i,j)(\exists\alpha\in S)\;\tau\preceq \lim_{s\to\infty}\gamma^{e,i,j}_s(\alpha)\}.
\]
In particular, the $\emptyset$-th fiber $\mathcal{T}$ of our compact set will be of the following form:
\[\mathcal{T}=\mathcal{S}\cup\left\{\lim_\ell\lim_s\gamma^{e,i,j}_s(x\upr \ell):e\in\om\mbox{ and }x\in\mathcal{S}\right\}.\]
Note that our requirement $\mathcal{N}_e$ will be ensured in the following way: 
\[\mathcal{N}_{e,i,j}:\;(\forall x\in\mathcal{S})\;f_{\Lambda_e}(\lim_\ell\lim_s\gamma^{e,i,j}_s(x\upr \ell))\not\in\mathcal{U}_{i,j}.\]
This construction will automatically ensure that $\mathcal{S}\subseteq\mathcal{T}$.
Therefore, the identity map witnesses the assertion (1) in Theorems \ref{thm:main} and \ref{thm:main3}.

\subsection*{Priority Tree}

To simplify our notations, we first fix $e,i,j$, and remove $e,i,j$ from superscripts and subscripts, e.g., hereafter $\gamma^{e,i,j}_s$ and $\rho_{e,i,j}$ will be denoted by $\gamma_s$ and $\rho$ respectively.
Let $\Lambda=(V,{\rm rk},p,q,\eta,f)$ be the weak-totalization of the $e$-th partial computable flow on a fixed vein $\mathbb{V}=(\mathcal{V},{\rm rk}_\mathcal{V})$ (we also abbreviate all superscripts and subscripts in $\Lambda^{\rm tot}_e$ for notational convenience).
We call $V$ a {\em priority tree} (associated with the $\Lambda$-piecewise computation), and each $\xi\in V$ a {\em $\Lambda$-strategy} (or simply, a {\em strategy}).
At the beginning of stage $s$, we inductively assume that $T_s$ has already been constructed, where $T_0$ is the tree consisting of all strings comparable with $\rho=\rho_{e,i,j}$.
Let $\Lambda_s=(V_s,{\rm rk}_s,p,q,\eta_s,f_s)$ be the stage $s$ approximation of the weakly total flow $\Lambda$.

For each string $\sigma\in T_s$ of length $s$, we will inductively define the {\em current true path ${\rm tp}_\Lambda(\sigma)\in V_s$} and for each strategy $\xi\in V$, the {\em $\xi$-timer} $t_\Lambda(\xi,\sigma)\in\om$.
On the root $\pair{}$, reset the $\xi$-timer to be $t_\Lambda(\xi,\pair{})=0$ for each strategy $\xi\in V$.
Assume that the current true path ${\rm tp}_\Lambda(\sigma^-)\in V_s$ and the $\xi$-timer $t_\Lambda(\xi,\sigma^-)$  for every strategy $\xi\in V$ has been already defined.
Assume inductively that $\xi:={\rm tp}_\Lambda(\sigma)\upr n$ has already been produced.
\begin{enumerate}
\item If the computation ${\rm rk}(\xi)$ does not converge by stage $s$, then recall that $\xi$ has no successor in $V_s$.
Then we define ${\rm tp}_\Lambda(\sigma)=\xi$.
\item If $\xi\not\in V^{\rm leaf}$ and ${\rm rk}(\xi)=2$, the {\em outcome} of the strategy $\xi$ is the least $i$ such that $p(\xi,i,\sigma\upr t_\Lambda(\xi,\sigma^-))$, and define ${\rm tp}_\Lambda(\sigma)\upr n+1=\xi\fr i$.
\item If $\xi\not\in V^{\rm leaf}$ and ${\rm rk}(\xi)= 1$, the outcome of the strategy $\xi$ is the least $i$ such that $q(\xi,i,\sigma\upr j)$ for every $j<s$, and define ${\rm tp}_\Lambda(\sigma)\upr n+1=\xi\fr i$.
\item If $\xi\not\in V^{\rm leaf}$ and ${\rm rk}(\xi)=0$, there are two cases:
\begin{enumerate}
\item If there is $i\leq s$ such that the computation of $\eta(\xi,i)$ converges by stage $s$ and $\sigma\succeq\eta(\xi,i)$, then the outcome of $\xi$ is the least such $i$, and define ${\rm tp}_\Lambda(\sigma)\upr n+1=\xi\fr i$.
\item Otherwise, define ${\rm tp}_\Lambda(\sigma)=\xi$.
\end{enumerate}
\item If $\xi$ is a leaf of $V$, then define ${\rm tp}_\Lambda(\sigma)=\xi$.
\end{enumerate}

For (2) and (3), we note that such $i$ must exist by weak-totality of $\Lambda$.
If $\xi\preceq{\rm tp}_\Lambda(\sigma)$, then we say that {\em $\xi$ is eligible to act along $\sigma$}, and the strategy $\xi$ sets the $\xi$-timer ahead by one second, i.e., $t_\Lambda(\xi,\sigma)=t_\Lambda(\xi,\sigma^-)+1$.
Otherwise, put $t_\Lambda(\xi,\sigma)=t_\Lambda(\xi,\sigma^-)$.

\begin{lemma}\label{lem:leftmost}
${\rm TP}_\Lambda(x)=\liminf_{n\to\infty}{\rm tp}_\Lambda(x\upr n)$, that is, ${\rm TP}_\Lambda(x)$ is the leftmost leaf of $S$ that is eligible to act along $x\upr n$ for infinitely many $n$.
\qed
\end{lemma}

Obviously, $\sigma\mapsto{\rm tp}_\Lambda(\sigma)$ is computable.
Therefore, this is an effective procedure approximating the piecewise computation induced by the flow $\Lambda$.

\begin{definition}[Priority-Value]
Given $\sigma\in T_{|\sigma|}$, the {\em priority value} ${\rm prior}_\Lambda(\xi,\sigma)$ of $\xi\in V$ along a string $\sigma$ is defined as follows:
\[{\rm prior}_\Lambda(\xi,\sigma)=\sum_{s=0}^{|\sigma|}\#\{\zeta<_{\rm left}\xi:\zeta\mbox{ is eligible to act along }\sigma\upr s\}.\]
\end{definition}\medskip

\noindent In other words, the current true path ${\rm tp}_\Lambda(\sigma)$ forces all strategies strictly to the right of ${\rm tp}_\Lambda(\sigma)$ to increase their priority values.
Clearly, if a strategy is an initial segment of the true path, then its priority value converges to some finite number.
It is also not hard to see the following:

\begin{lemma}\label{lem:smooth}
Consider the following partial functions $\tilde{s},\hat{s}$:
Let $n\in\om$, $x\in\om^\om$, and a strategy $\xi\in V$ be given (as inputs).
For any $u\in\om$ and any strategy $\zeta$ which is eligible to act along $x\upr u$,
\begin{enumerate}
\item if $u\geq\tilde{s}(\xi,x,n)$, and if $\xi<_{\rm left}\zeta$, then the priority value of $\zeta$ must be greater than $n$,
\item if $u\geq \hat{s}(x,n)$, either $\zeta$ is an initial segment of the true path or else the priority value of $\zeta$ must be greater than $n$.
\end{enumerate}

\noindent Then, $\hat{s}:2^\om\times\om\to\om$ is total, and $\tilde{s}(\xi,x,n)$ is defined for all $x\in 2^\om$ and $n\in\om$ whenever $\xi$ is an initial segment of the true path.
Moreover, $\tilde{s}$ is computable.
\end{lemma}

\begin{proof}
There are infinitely many stages $(t_i)_{i\in\om}$ such that the true path is eligible to act along $x\upr t_i$.
Let $t_n>s_0$ be the $n$-th such stage.
Then, after stage $t_n$, the priority value of any strategy strictly to the right of the true path becomes greater than $n$ by definition.
We now define the function $\tilde{s}$ as follows: given $\xi,x,n$, we wait for seeing the $n$-th stage $u_n$ at which the strategy $\xi$ is eligible to act along $x$, and define $\tilde{s}(\xi,x,n)=u_n$ (this value may not be defined if $\xi$ is not an initial segment of the true path).
This procedure is computable.
If $\xi$ is an initial segment of the true path, then $\tilde{s}(\xi,x,n)=t_n$ which satisfies the desired condition.
To define $\hat{s}$, by using Lemma \ref{lem:leftmost}, we choose a stage $s_0$ such that no strategy strictly to the left of the true path is eligible to act along $x\upr s_0$.
We also choose $\xi$, an initial segment of the true path.
Then, if $t\geq\max\{s_0,\tilde{s}(\xi,x,n)\}$, $\zeta$ is eligible to act along $x\upr t$, and the priority value of $\zeta$ is not greater than $n$, then $\zeta$ must be an initial segment of the true path.
Consequently, $\hat{s}(x,n)=\max\{s_0,\tilde{s}(\xi,x,n)\}$ satisfies the desired condition.
\end{proof}


\subsection*{Construction}

We first put $\gamma_0(\alpha)=\rho\fr \alpha$ for every $\alpha\in S$.
Assume that $\gamma_s(\alpha)$ (and hence $T_s$) has already been constructed.

We say that $\alpha\in S$ is {\em active at stage $s$} if the length of $\gamma_s(\alpha)$ is at most $|\rho|+s$.
In other words, $\gamma_s(\alpha)$ is an initial segment of a string of $T^\ast_s$, where $T^\ast_s$ is the set of all strings in $T_s$ of length $|\rho|+s$.
Given $\sigma\in T_s^\ast$, let $\gamma_s^\leftarrow(\sigma)$ denote the maximal string $\alpha\in S$ such that $\gamma_s(\alpha)\preceq\sigma$.
Then the set of all active strings of $S$ (with respect to a fixed triple $e,i,j$) at stage $s$ can be written as follows:
\[S_s=\{\alpha\in S:(\exists\sigma\in T^\ast_s)\;\alpha\preceq\gamma^{\leftarrow}_s(\sigma)\}.\]

For a leaf $\xi\in V^{\rm leaf}$, if the immediate predecessor $\xi^-$ is finitely branching, then we define ${\xi}^\ast=\xi^-$; otherwise, we define ${\xi}^\ast=\xi$.
A strategy is said to be {\em almost-terminal} if it is of the form $\xi^\ast$ for some leaf $\xi\in V^{\rm leaf}$.
We write $V^{\rm leaf}_\xi$ as the set of leaves in $V$ extending $\xi$.
Note that $V^{\rm leaf}_\xi$ is finite for any almost-terminal strategy $\xi$.

We now see that each $\sigma\in T_s^\ast$ is layered as
\[\gamma_s(\gamma_s^{\leftarrow}(\sigma)\upr 0)\prec \gamma_s(\gamma_s^{\leftarrow}(\sigma)\upr 1)\prec\dots\prec \gamma_s(\gamma_s^{\leftarrow}(\sigma))\preceq\sigma\]

Given $\sigma$, an almost-terminal strategy $\xi$ calculates the priority value $p={\rm prior}(\xi,\sigma)$, and then monitors the $p$-th level of the above layer.
The almost-terminal strategy $\xi$ is allowed to extend the $p$-th level string as $\gamma_{s+1}(\gamma_s^{\leftarrow}(\sigma)\upr p)\succ \gamma_{s}(\gamma_s^{\leftarrow}(\sigma)\upr p)$.
Such an action may injure all lower priority strategies.
Formally speaking, we say that an almost-terminal strategy $\xi\in V$ is {\em active along $\sigma$} if $\xi$ is eligible to act along $\sigma$ and moreover, its priority value ${\rm prior}(\xi,\sigma)$ is less than or equal to the length of $\gamma_s^{\leftarrow}(\sigma)$.
Then we also say that $\xi$ {\em monitors $\alpha$ along $\sigma$} if $\xi$ is active along $\sigma$, and if
\[\alpha=\gamma_s^{\leftarrow}(\sigma)\upr {\rm prior}(\xi,\sigma).\]


\subsubsection*{Strategy}
Stage $s$ has substages $t\leq u$ at which an almost-terminal strategy $\xi\in V$ of the priority value $t$ along some string may act, where $u$ is the length of a longest string in $S_s$.
We describe the action of our strategy at substage $t$.
\begin{enumerate}
\item We say that an almost-terminal strategy $\xi$ {\em requires attention at substage $t$} if
\begin{align*}
(\exists\sigma_\xi\in T_s^\ast)(\exists\alpha_\xi\in &S_s\cap 2^t)\;[\xi\mbox{ monitors $\alpha_\xi$ along $\sigma_\xi$},\\
&\mbox{ and }(\exists\lambda\in V^{\rm leaf}_\xi)\;f_\lambda(\gamma_s(\alpha_\xi))\prec f_\lambda(\sigma_\xi)\in U[s]],
\end{align*}
where recall that $U[s]$ is the stage $s$ approximation of $U^z_{i,j}$ for $z=\emptyset$ and fixed $i$ and $j$.
Recall also that $U:=\bigcup_sU[s]$ has no computable element.
\item If such a $\xi\in V$ exists, choose a strategy $\xi$ having the shortest $\gamma_s(\alpha_\xi)$ among strategies requiring attention at substage $t$.
Then we say that {\em $\xi$ receives attention along $\sigma_\xi$ at substage $t$}, and the strategy $\xi$ acts as follows:
\begin{enumerate}
\item Define $\gamma_{s+1}(\alpha_\xi)$ to be such a string $\sigma_\xi\in T_s$.
\item Then, injure all lower priority constructions by defining $\gamma_{s+1}(\alpha_\xi\fr\beta)=\gamma_{s+1}(\alpha_\xi)\fr\beta$ for every $\beta$ such that $\alpha_\xi\fr\beta\in S$.
For a string $\beta\in S$ which does not extend $\alpha_\xi$, we define $\gamma_{s+1}(\beta)=\gamma_{s}(\beta)$.
We have $\gamma_{s+1}(\beta)\succeq\gamma_s(\beta)$ unless $\beta\succ\alpha_\xi$.
\item Skip all substrategies after $t+1$, and go to substage $0$ of stage $s+1$.
\end{enumerate}
\item If there is no such $\xi$ and if $t<u$, go to substage $t+1$.
If $t=u$, then define $\gamma_{s+1}(\alpha)=\gamma_s(\alpha)$ for all $\alpha\in S$ and go to substage $0$ of stage $s+1$.
\end{enumerate}\medskip

\noindent By our construction, it is clear that $\sigma\in T_s^\ast$ implies $\sigma\fr\tau\in T_s$ for all $\tau\in 2^{<\omega}$.
Therefore, this construction ensures that $\gamma_{s+1}(\alpha)\in T_s$ for all $\alpha\in S$.

\begin{lemma}\label{lem:prio-converge}
Every strategy requires attention at most finitely often.
Therefore, $\lim_s\gamma(\alpha,s)$ converges for every $\alpha\in S$.
\end{lemma}

\begin{proof}
For $\alpha\in S$, inductively assume that we have already shown that $\lim_s\gamma_s(\beta)$ converges for any initial segment $\beta\prec\alpha$.
Let $s_0$ be the least stage such that $\gamma_s(\beta)=\gamma_{s_0}(\beta)$ for all $s\in\om$ and $\beta\prec\alpha$.
The item (2-b) in our construction ensures that $\gamma_s(\alpha)$ is monotone after stage $s_0$.
Suppose for the sake of contradiction that $x=\lim_s\gamma_s(\alpha)$ does not converge, that is, $x$ is an infinite string.
By monotonicity of $\gamma_s(\alpha)$ after stage $s_0$ and effectivity of our construction, it is not hard to see that $x$ is computable.

We first note that if $\gamma_{s+1}(\alpha)\succ\gamma_s(\alpha)$ happens for some $s\geq s_0$, then this change is caused by a strategy $\xi$ monitoring $\alpha$ along some string $\sigma_\xi$ at stage $s$.
Moreover, this $\sigma_\xi$ must be an initial segment of $x$.
Otherwise, the change $\gamma_{s+1}(\alpha)\succ\gamma_s(\alpha)$ is caused by an action of a strategy along $\sigma_\xi\not\prec x$, and then the strategy requires $\gamma_{s+1}(\alpha)$ to become $\sigma_\xi\not\prec x$, which contradicts monotonicity of $\gamma_s(\alpha)$ after stage $s_0$.

Let $\xi:={\rm TP}_\Lambda(x)$ be the true path through $\Lambda$ along $x$.
By our construction, only strategies $\xi$ with ${\rm prior}(\xi,x\upr s)\leq |\alpha|$ can change the value $\gamma_s(\alpha)$.
Let $\hat{s}(x,|\alpha|)$ be a stage in Lemma \ref{lem:smooth}, and let $s_1$ be the maximum of $s_0$ and $\hat{s}(x,|\alpha|)$.
Then, if $\gamma_s(\alpha)$ changes after stage $s_1$, this change must be caused by an initial segment $\xi$ of the true path ${\rm TP}_\Lambda(x)$.
Clearly, there is a unique almost-terminal strategy $\xi$ which is an initial segment of the true path ${\rm TP}_\Lambda(x)$.

Let $s(n)\geq s_1$ be the $n$-th stage such that $\gamma_{s(n)+1}(\alpha)\succ\gamma_{s(n)}(\alpha)$ happens.
In this case, the unique almost-terminal strategy $\xi\prec{\rm TP}_\Lambda(x)$ requires attention at substage $|\alpha|$ of stage $s(n)$ with witnesses $\sigma_\xi\prec x$ and $\alpha$, and therefore, there is a leaf $\lambda(n)\in V$ extending $\xi$ such that $f_{\lambda(n)}(\gamma_{s(n)}(\alpha))\prec f_{\lambda(n)}(\sigma_\xi)\in U$.
Moreover, since the change of $\gamma_s(\alpha)$ is caused by this strategy $\xi$, it must receive attention along $\sigma_\xi$, and therefore $\gamma_{s(n)+1}(\sigma_\xi)=\sigma_\xi$.
Hence, we have $f_{\lambda(n)}(\gamma_{s(n)}(\alpha))\prec f_{\lambda(n)}(\gamma_{s(n)+1}(\alpha))\in U$ for every $n\in\om$.
By the definition of being almost-terminal, there are only finitely many leaves in $V$ extending $\xi$.
Therefore, by the pigeonhole principle, there is a leaf $\lambda\in V$ extending $\xi$ such that $f_{\lambda}(\gamma_{s(n)}(\alpha))\prec f_{\lambda}(\gamma_{s(n)+1}(\alpha))\in U$ for infinitely many $n\in\om$.
%
By monotonicity of $\gamma$ after stage $s_0$ and the above property, $f_\lambda(x)$ produces an infinite string and $f_\lambda(x)\in\mathcal{U}$.
Since $f_\lambda$ and $x$ are computable, $f_\lambda(x)$ is a computable element of $\mathcal{U}$.
However, this contradicts our assumption that $\mathcal{U}$ has no computable element.
\end{proof}

Hereafter we write $\gamma(x)=\lim_\ell\lim_s\gamma_s(x\upr \ell)$.
Recall that every $y\in\mathcal{T}$ extending $\rho$ is of the form $\gamma(x)$ for some $x\in\mathcal{S}$.
By Lemma \ref{lem:prio-converge}, given $y\in T$ and $n\in\om$, for any sufficiently large $\ell_n$ and $s_n$ such that $\gamma^{\leftarrow}_{s_n}(y\upr\ell_n)\upr n$ is uniquely determined.
Therefore, we define $\gamma^{\leftarrow}(y)$ as an infinite string satisfying $\gamma^{\leftarrow}(y)\upr n=\gamma^{\leftarrow}_{s_n}(y\upr \ell_n)\upr n$.
It is not hard to see that $\gamma^\leftarrow(\gamma(x))=x$ for every $x\in\mathcal{S}$.

\begin{lemma}
$f_\Lambda(\gamma(x))\not\in\mathcal{U}$ for every $x\in\mathcal{S}$.
\end{lemma}

\begin{proof}
Otherwise, there is $x\in\mathcal{S}$ such that $f_{\Lambda_e}(\gamma(x))\in\mathcal{U}$.
Recall that $\Lambda$ is the weak-totalization of the $e$-th partial computable flow $\Lambda_e$; therefore $\Lambda$ is equivalent to $\Lambda_e$.
In particular, $f_\Lambda(\gamma(x))=f_{\Lambda_e}(\gamma(x))$ holds.
We denote by $\xi$ the true path ${\rm TP}_\Lambda(\gamma(x))$ along $\gamma(x)$, and then we have $f_\Lambda(\gamma(x))=f_\xi(\gamma(x))$.
Since $\xi$ is the true path along $\gamma(x)$, there is stage $s_0$ such that the priority value of $\xi$ along $\gamma(x)$ converges to some $p\in\om$, that is, ${\rm prior}(\xi,\gamma(x)\upr s)=p$ for all $s\geq s_0$.
By Lemma \ref{lem:prio-converge} there is stage $s_1\geq s_0$ such that $\gamma_u(\alpha)=\gamma_{s_1}(\alpha)$ for any $\alpha\in S$ of length at most $p$ and any stage $u\geq s_1$.
In particular, no strategy which monitors $\alpha\in S$ of length at most $p$ along some string receives attention at substage $t\leq p$ of stage after $s_1$.
For $\alpha=\gamma^{\leftarrow}(\gamma(x))\upr p$, note that $\gamma_u(\alpha)$ for any $u\geq s_1$ is of the form $\gamma(x) \upr \ell$ for some $\ell\in\om$.
However, if $f_\xi(\gamma(x))\in\mathcal{U}$ then for any $\ell\in\om$ there is $u\geq s_1$ such that $f_\xi(\gamma(x)\upr \ell)\prec f_\xi(\gamma(x)\upr v)\in U[v]$ for all $v\geq u$.
Moreover, $\xi$ is eligible to act at some stage $v\geq u$ since $\xi$ is the true path along $\Lambda$ at $\gamma(x)$.
Therefore, some strategy must receive attention at substage $\leq p$ of such stage $v\geq u\geq s_1$, which is a contradiction because of our choice of $s_1$.
\end{proof}

Finally, we direct our attention to the global requirement $\mathcal{G}$, and therefore, we have to analyze the whole picture of the $\emptyset$-th fiber of $\mathcal{T}$.
To see the property of $\mathcal{T}(\emptyset)$ we now need to restore the subscripts and superscripts such as $e,i,j$.
For instance, we consider $\gamma_{e,i,j}(x)=\lim_\ell\lim_s\gamma^{e,i,j}_s(x\upr \ell)$ and $\gamma^{\leftarrow}_{e,i,j}$ defined as above.
Recall that the $\emptyset$-th fiber of $\mathcal{T}$ is defined as $\mathcal{S}(\emptyset)\cup\{\gamma_{e,i,j}(x):e,i,j\in\om\mbox{ and }x\in\mathcal{S}(\emptyset)\}$.
To make sure that the global requirement $\mathcal{G}$ is satisfied, we will construct a $\mathbb{V}'$-piecewise computable function $g_\emptyset:\mathcal{T}(\emptyset)\to\mathcal{S}(\emptyset)$.
We consider the following function:
\[
g_\emptyset(y)=
\begin{cases}
y & \mbox{if }y\in\mathcal{S}(\emptyset),\\
x & \mbox{if }y\mbox{ is of the form $\gamma_{e,i,j}(x)$}.
\end{cases}
\]

\begin{lemma}\label{lem:global-requirement}
$g_\emptyset$ is $\mathbb{V}'$-piecewise computable.
\end{lemma}

\begin{proof}
In this proof, to avoid confusion, we use the symbols $\lceil m_0,\dots,m_n\rceil$ to denote a natural number coding the tuple $\langle m_0,\dots,m_n\rangle$.
We need to construct a flow $\Lambda'$ on the vein $\mathbb{V}'$ such that $f_{\Lambda'}$ is equal to $g_\emptyset$.
Recall that $\mathbb{V}'$ is of the form $\mathbb{V}^{\ominus 1\oplus_\om 0\oplus 1}$ if the root of $\mathbb{V}$ is finitely branching; otherwise, it is of the form $\mathbb{V}^{\ominus 1\oplus k}$, where $k={\rm rk}(\langle\rangle)+1$.
Thus, all new nodes in $\mathbb{V}'$ are infinitely branching except for the root.
To define the flow, our branching function $b'$ first converts the root of $\mathbb{V}$ into a $2$-branching node, and put the $1$-replacement of the $(e,i,j)$-th flow below $\langle 1,\lceil e,i,j\rceil\rangle$.
That is, the branching function $b'$ is defined by $b'(\langle\rangle)=2$ and $b'(\langle k,\lceil e,i,j\rceil\rangle\fr\xi)=b_e(\xi)$ for all $k,e,i,j\in\om$ and $\xi\in \mathbf{V}_e$.
Here, if the root of $\mathbf{V}$ is infinitely branching, then hereafter we think of $\langle 1,\lceil e,i,j\rceil\rangle\fr\xi$ as an abbreviation of $\langle 1,\lceil e,i,j,n\rceil\rangle\fr\zeta$, where $\xi=n\fr\zeta$.

If the root of $\mathbf{V}$ is finitely branching, we define the labeled well-founded tree $\mathbf{V}'$ as the result after removing all extensions of $\langle 0\rangle$ from $(\mathbb{V}')^{b'}$.
More explicitly, we may define the labeled well-founded tree $\mathbf{V}'=(V',{\rm rk}')$ on the vein $\mathbb{V}'$ as follows:
\begin{align*}
V'=\{\langle\rangle,\langle 0\rangle,\langle 1\rangle\}&\cup\{\langle 1,\lceil e,i,j\rceil\rangle\fr\xi:e,i,j\in\om\mbox{ and }\xi\preceq\zeta\mbox{ for some }\zeta\in V_e^{\rm at}\}\\
&\cup\{\langle 1,\lceil e,i,j\rceil\rangle\fr\xi\fr n:\xi\in V_e^{\rm at}\mbox{ and }n\in\om\},
\end{align*}
where recall that $V_e^{\rm at}$ is the set of all almost-terminal strings in $V_e$, and ${\rm rk}'$ is defined in a straightforward manner.
If the root of $\mathbf{V}$ is infinitely branching, we convert $\langle 0\rangle$ into a two-branching Borel-rank-$1$ node such that the lefthand side $\langle 00\rangle$ is a leaf and the righthand side $\langle 01\rangle$ is an infinitely-branching Borel-rank-$0$ node all of whose successors are leaves.
Formally speaking, we define the labeled well-founded tree $\mathbf{V}'=(V',{\rm rk}')$ on the vein $\mathbb{V}'$ as follows:
\begin{align*}
V'=\{\langle\rangle,\langle &0\rangle,\langle 1\rangle,\langle 00\rangle,\langle 01\rangle\}\cup\{\langle 1,\lceil e,i,j\rceil\rangle\fr\xi:e,i,j\in\om\mbox{ and }\xi\preceq\zeta\mbox{ for some }\zeta\in V_e^{\rm at}\}\\
&\cup\{\langle 01e\rangle:e\in\om\}\cup\{\langle 1,\lceil e,i,j\rceil\rangle\fr\xi\fr n:\xi\in V_e^{\rm at}\mbox{ and }n\in\om\},
\end{align*}
and ${\rm rk}'(\langle\rangle)={\rm rk}(\langle\rangle)+1$, ${\rm rk}'(\langle 0\rangle)=1$, and ${\rm rk}'(\langle 01\rangle)=0$; the other values of ${\rm rk}'$ are defined in a  straightforward manner.

We place a flowchart $\Lambda'$ on the labeled well-founded tree $(V',{\rm rk}')$.
We first assume that the root of $\mathbb{V}$ is finitely branching, and therefore $\mathbb{V}'$ is of the form $\mathbb{V}^{\ominus 1\oplus_\om 0\oplus 1}$.
Recall that the root $\langle\rangle$ branches into two nodes in $V'$, and the Borel rank ${\rm rk}'(\langle\rangle)$ of this branching is $1$.
We put the $\Pi^0_1$-branch on the root $\langle\rangle$ which asks whether a given input $x\in 2^\om$ extends $\rho_{e,i,j}$ for some $e,i,j$ or not.
Formally speaking, the first $\Pi^0_1$-branching condition is given as follows:
\[
(\forall\sigma\in 2^{<\om})\quad q'(\langle\rangle,0,\sigma)\;\iff\;(\forall e,i,j\in\om)\;\rho_{e,i,j}\not\preceq\sigma,
\]
and $q'(\langle\rangle,1,\sigma)$ is a formula which is always true for any $\sigma\in 2^{<\om}$.

Next, recall that if we answer yes to this first $\Pi^0_1$-question on the root $\langle\rangle$, then the computation directs into the left node $\langle 0\rangle$, which is a terminal node in $V'$; therefore, some continuous function $f_{\langle 0\rangle}$ has to be placed on this terminal node $\langle 0\rangle$.
We describe the following instruction in our flowchart $\Lambda'$:
If we answer yes to the first $\Pi^0_1$-question with a given input $x$, that is, if $q'(\langle\rangle,0,x\upr n)$ for all $n\in\om$, then return $x$ itself.
In other words, we define $f_{\langle 0\rangle}$ to be the identity function.

Recall also that if we answer no to the first $\Pi^0_1$-question on the root $\langle\rangle$, then the computation directs into the right node $\langle 1\rangle$, which is an infinitely branching node in $V'$ with Borel rank $0$; therefore, a $\Delta^0_0$-conditional branch (given by a collection $(\eta'(\langle 1\rangle,\lceil e,i,j\rceil))_{e,i,j\in\om}$ of binary strings) has to be placed on the node $\langle 1\rangle$.
We describe the following instruction in our flowchart $\Lambda'$:
If we answer no to the first $\Pi^0_1$-question with a given input $x$, then consider the $\Delta^0_0$-question which asks what the least number $\lceil e,i,j\rceil$ such that $x$ extends the string $\rho_{e,i,j}$ is.
In other words, we define $\eta'(\langle 1\rangle,\lceil e,i,j\rceil)$ to be $\rho_{e,i,j}$.

If we answer $(e,i,j)$ to the $\Delta^0_0$-question on $\langle 1\rangle$, recall that the shape of the labeled well-founded tree $(V',{\rm rk}')$ below $\langle 1,\lceil e,i,j\rceil\rangle$ is exactly the same as the labeled well-founded tree $(V^{\rm tot}_e,{\rm rk}^{\rm tot}_e)$ in the weak-totalization of the $e$-th flowchart $\Lambda_e$ except that each almost-terminal node $\xi$ in $V_e$ becomes an infinitely branching node $\langle 1,\lceil e,i,j\rceil\rangle\fr \xi$ in $V'$ with Borel rank $1$.
Thus, if a node is of the form $\langle 1,\lceil e,i,j\rceil\rangle\fr\zeta$ for some $\zeta\in V_e$ which does not reach an almost-terminal node, then we put the same question on $\langle 1,\lceil e,i,j\rceil\rangle\fr\zeta$  in our flowchart $\Lambda'$ as that on $\zeta$ in the $e$-th flowchart $\Lambda_e$.

Now, for any almost-terminal node $\xi$ in $V_e$, we need to put a new $\Pi^0_1$-branching condition on $\langle 1,\lceil e,i,j\rceil\rangle\fr \xi$ and a new continuous function on each leaf $\langle 1,\lceil e,i,j\rceil\rangle\fr\xi\fr k$.
The {\em length-of-agreement} of a leaf $\lambda$ in the original tree $V_e$ with respect to a tree $U^\emptyset_{i,j}$ is defined as follows:
\[\ell_{\lambda}^{i,j}(\sigma)=\max\{n\in\omega:f_\lambda(\sigma;m)\downarrow
\mbox{ for every $m<n$, and }f_\lambda(\sigma)\upr n\in
U^\emptyset_{i,j}\}.\]\medskip

\noindent For an almost-terminal node $\xi$ in $V_e$, if a computation reaches the node $\langle 1,\lceil e,i,j\rceil\rangle\fr\xi$ in $V'$ along our flowchart, consider the following Borel-rank-$1$ question:
Give the pair $(s,n)$ satisfying the following:
\begin{enumerate}[label=(\Roman*)]
\item $s$ is the least stage after which no $\Lambda_e$-strategy $\zeta<_{\rm left}\xi$ acts,
\item and $n$ is the total number of lengths-of-agreement of leaves $\lambda$ in the original tree $V_e$ extending $\xi$.
\end{enumerate}
Here recall that an almost-terminal node has only finitely many successors, and therefore, the sum of lengths-of-agreement must be finite.
Formally speaking, if $\xi$ is an almost-terminal node in the original tree $V_e$, then we consider the following $\Pi^0_1$ formula on the node $\langle 1,\lceil e,i,j\rceil\rangle\fr \xi$ in $V'$:
\begin{align*}
q'(\langle{1,\lceil e,i,j\rceil}\rangle\fr \xi,\lceil s,n\rceil,\sigma)\;\iff\;&\;(\forall t)\;[(s<t<|\sigma|)\;\rightarrow\;(\neg\exists\zeta<_{\rm left}\xi)\;\zeta\preceq{\rm tp}_{\Lambda_e}(\sigma\upr t)]\\
&\mbox{and }\ell^{i,j}_\xi(\sigma):=\sum \left\{\ell_{\lambda}^{i,j}(\sigma):\rho\in V_e^{\rm leaf}\mbox{ and }\xi\preceq\lambda\right\}\leq n,
\end{align*}
It is not hard to check that this formula is $\Pi^0_1$.

Finally, we need to place a continuous function $f'_\theta$ on each leaf $\theta:=\langle 1,\lceil e,i,j\rceil\rangle\fr\xi\fr\lceil s,n\rceil$ in $V'$, where $\xi$ is an almost-terminal node in $V_e$.
The function $f'_\theta$ will act under the belief that $\xi$ is an initial segment of the true path through $\Lambda_e^{\rm tot}$, and $s$ and $n$ are the correct values satisfying the above conditions (I) and (II).
In particular, $f'_\theta(x)$ always believes that $\langle{1,\lceil e,i,j\rceil}\rangle\fr\xi\fr\lceil s,n\rceil$ is the true path through our new flowchart $\Lambda'$ along any given input $x$.
Given an input $x\in 2^\om$, the value $f'_\theta(x)$ is computed in the following manner:
\begin{enumerate}
\item First, calculate the priority value $p:={\rm prior}_{\Lambda^{\rm tot}_e}(\xi,\sigma)$ of $\xi$ along $x\upr s$.
\item Next, wait for seeing stage $s_0\geq s$ such that the sum of lengths-of-agreement has become $n$, and moreover $\xi$ has already received attention because of the change of the total value of lengths-of-agreement to $n$ by stage $s_0$.
Formally speaking, for any stage $u\in\om$ and any leaf $\lambda\in (V^{\rm tot}_e)^{\rm leaf}_\xi$, consider the value $v(\lambda,u)$ defined by the maximal length of $f_\lambda(\sigma_\xi)$ such that $\xi$ receives attention along $\sigma_\xi$ witnessed by $\rho$ at a substage of some stage $s'<u$.
Then, wait for stage $s_0\geq s$ such that we see the following equations:
\[\sum\{v(\lambda,s_0):{\lambda\in V^{\rm leaf}_\xi}\}=\ell^{i,j}_\xi(x\upr s_0)=n.\]
\item Moreover, for any $\ell\geq p$, wait for seeing stage $s(\ell)=\max\{\tilde{s}(\xi,x,\ell),s_0\}$ where $\tilde{s}$ is the partial computable function in  Lemma \ref{lem:smooth}.
\item For a given $\ell\geq p$, search for the unique $\alpha_\ell$ of length $\ell$ such that $\gamma_{s(\ell)}(\alpha_\ell)\preceq x$, and then return $f_\theta(x)(\ell-1)=\alpha_\ell(\ell-1)$.
Such $\alpha_\ell$ exists whenever $x\in\mathcal{T}$.
\end{enumerate}
Note that $f_\theta$ is partially  computable uniformly in $e,i,j,\xi,s,n$.

\begin{claim}
If $x$ is an infinite path through $T(\emptyset)$ and $\theta=\langle{1,\lceil e,i,j\rceil}\rangle\fr\xi\fr\lceil s,n\rceil$ is the true path through our new flowchart $\Lambda'$ along an input $x$, then $f_\theta(x)=\gamma^{\leftarrow}_{e,i,j}(x)$ holds.
\end{claim}

\noindent
It is not hard to see that if $\langle{1,\lceil e,i,j\rceil}\rangle\fr\xi\fr\lceil s,n\rceil$ is the true path then $(s,n)$ is the correct pair satisfying the above conditions (I) and (II).
Therefore, the priority value of $\xi$ never changes after stage $s$, that is, the value stabilizes to $p$, and the almost-terminal strategy $\xi$ never receives attention after stage $s_0$ where $s_0$ is the stage in (2).
As mentioned in the second paragraph in the proof of Lemma \ref{lem:prio-converge}, if an action of a strategy causes $\gamma^{e,i,j}_{s+1}(\alpha_\ell)\succ\gamma^{e,i,j}_s(\alpha_\ell)$ for some stage $s\geq s(\ell)$, then this strategy must act along an initial segment of $x$ unless $x\not\in\mathcal{T}(\emptyset)$ since $\gamma^{e,i,j}_{s(\ell)}(\alpha_\ell)$ is an initial segment of $x$.
However, no strategy can act along $x$ after stage $s(\ell)$.
Therefore, it has to be true that, for any stage $u\geq s(\ell)$, no strategy can change the value $\gamma^{e,i,j}_u(\alpha_\ell)$ along $x$.
In other words, we have $\lim_s\gamma^{e,i,j}_s(\alpha_\ell)=\gamma^{e,i,j}_{s(\ell)}(\alpha_\ell)\prec x$.
By uniqueness of such $\alpha_\ell$, we must have $\gamma^{\leftarrow}_{e,i,j}(x)\upr \ell=\alpha_\ell$.
Consequently, we get that $f_\theta(x)=\gamma^{\leftarrow}_{e,i,j}(x)$.

\medskip

We next consider the case that the root of $\mathbb{V}$ is infinitely branching.
In this case, we first need to put a ${\Pi}^0_{{\rm rk}(\langle\rangle)+1}$-question on the root.
If ${\rm rk}(\langle\rangle)=1$, then we put the following ${\Pi}^0_2$-question:
\[p'(\langle\rangle,0,\sigma)\;\iff\;\neg(\exists e,i,j,n\in\om)\;[\rho_{e,i,j}\preceq\sigma\mbox{ and }q_e(\langle\rangle,n,\sigma)],\]
and if ${\rm rk}(\langle\rangle)=0$, then we put the following ${\Pi}^0_1$-question:
\[q'(\langle\rangle,0,\sigma)\;\iff\;\neg(\exists e,i,j,n\in\om)\;[\rho_{e,i,j}\preceq\sigma\mbox{ and }\eta_e(\langle\rangle,n)\preceq\sigma]\]
In other words, these questions ask whether it is true that if a given input $x$ extends $\rho_{e,i,j}$ for some $e,i,j$ then the first outcome along the weak-totalization of $\Lambda_e$ is $0$, that is, $x\not\in{\rm dom}(\Lambda_e)$.
If we answer yes to this question, then the computation reaches the two-branching Borel-rank-$1$ node $\langle 0\rangle$.
We put the ${\Pi}^0_1$-question on $\langle 0\rangle$ asking whether a given input extends $\rho_{e,i,j}$ for some $e,i,j$ or not:
\[q'(\langle 0\rangle,0,\sigma)\;\iff\;(\forall e,i,j\in\om)\;\rho_{e,i,j}\not\preceq\sigma,\]
and then we put the identity function on the leaf $\langle 00\rangle$ as in the finitely-branching case.
Similarly, we put the rank-$0$ question on the infinitely-branching Borel-rank-$0$ node $\langle 01\rangle$ asking for the unique triple $(e,i,j)$ such that a given input extends $\rho_{e,i,j}$, that is, define $\eta(\langle 01\rangle,\lceil e,i,j\rceil)=\rho_{e,i,j}$ for each $e,i,j\in\om$.
Now we need to put a continuous function $f_{01\lceil e,i,j\rceil}$ on the leaf $\langle 01\lceil e,i,j\rceil\rangle$ for each $e,i,j\in\om$.
The function believes that a given input $x$ extends $\rho_{e,i,j}$ and the first outcome along the weak-totalization of $\Lambda_e$ is $0$.
The function $f_{01\lceil e,i,j\rceil}$ proceeds similarly to the function $f_\theta$ described above except that we do not need to care about the length-of-agreement and the priority value, because the function $f_{01\lceil e,i,j\rceil}$ believes that $x\not\in{\rm dom}(\Lambda_e)$ and $\langle 0\rangle$ is always the leftmost path of $\Lambda_e^{\rm tot}$.
By the same argument as above, we can show that if $x$ is an infinite path through $T(\emptyset)$ and $\langle{01\lceil e,i,j\rceil}\rangle$ is the true path through our new flowchart $\Lambda'$ along an input $x$, then $f_{01\lceil e,i,j\rceil}(x)=\gamma^{\leftarrow}_{e,i,j}(x)$ holds.
All other parts in the infinitely-branching case are similar as the finitely-branching case except that we consider $\langle 1,\lceil e,i,j,n\rceil\rangle\fr\zeta$ instead of $\langle 1,\lceil e,i,j\rceil\rangle\fr\xi$ where $\xi=n\fr\zeta$.

We are now ready for proving that $f_{\Lambda'}$ agrees with $g_\emptyset$ on the domain $\mathcal{T}(\emptyset)$.
Given an input $x\in\mathcal{T}(\emptyset)$, consider the case that the first outcome of the true path through $\Lambda'$ along $x$ is $0$.
There are two cases.
First, consider the case that the root of the original vein $\mathbb{V}$ is finitely branching.
Then, if the first outcome through $\Lambda'$ is $0$ then we answer yes to the first $\Pi^0_1$-question on the root $\langle\rangle$, which means that $x$ does not extend $\rho_{e,i,j}$ for any $e,i,j\in\om$.
Therefore, we must have $x\in\mathcal{S}(\emptyset)$.
Since the identity function is placed on the left node $\langle 0\rangle$ in our flowchart $\Lambda'$, we have $f_{\Lambda'}(x)=x\in\mathcal{S}(\emptyset)$.
If the first two outcomes of the true path through $\Lambda'$ along $x$ are $\langle 1,\lceil e,i,j\rceil\rangle$, then the true path is of the form $\theta=\langle 1,\lceil e,i,j\rceil\rangle \fr \xi\fr \lceil s,n\rceil$ for some almost-terminal node $\xi$.
Then, by the above claim, we have $f_{\Lambda'}(x)=f_\theta(x)=\gamma^{\leftarrow}_{e,i,j}(x)\in\mathcal{S}$ as desired.

We can use a similar argument for the case that the root of the original vein $\mathbb{V}$ is infinitely branching.
\end{proof}

It is straightforward to relativize our argument to $(z\oplus d)$-computably construct each $z$-th fiber $\mathcal{T}(z)=\mathcal{S}(z)\cup\{\gamma^z_{e,i,j}(x):e,i,j\in\om\mbox{ and }x\in\mathcal{S}(z)\}$ to satisfy the requirements $(\mathcal{N}^z_{e,i,j})_{e,i,j}$ uniformly in $z\in 2^\om$.
Consequently, $\mathcal{T}=\{(z,x):x\in\mathcal{T}(z)\}$ is $\Pi^0_1(d)$.
Moreover, as in the above argument, for such $\mathcal{T}(z)$, one can uniformly construct a $\Lambda'$-piecewise continuous function $g_z:\mathcal{T}(z)\to\mathcal{S}(z)$ obtained from some flow $\Lambda_z'$ on the vein $\mathbb{V}'$. 
Note that our description of the flow $\Lambda'$ is effective, and by relativizing this, one can obtain a $d$-computable function producing such a flow $\Lambda_z'$ from $z\in 2^\om$.
It is not hard to check that this implies that the function $g$ defined by $g(z,x)=g_z(x)$ is $\mathbb{V}''$-piecewise $d$-computable.
Consequently, our construction fulfills the global requirement $\mathcal{G}$.
This concludes the proof of Theorems \ref{thm:main}, \ref{thm:main2} and \ref{thm:main3}.

\subsection*{Proof of \texorpdfstring{($\dagger$)}{dagger}}

Finally, we verify the statement ($\dagger$) claimed in Section \ref{sec:summary}.
As mentioned in Example \ref{example:Weihrauch-principle}, $\mathbb{V}_{1,1}$, $\mathbb{V}_{1,\om}$, and $\mathbb{V}_{2,1}$ correspond to ${\sf LPO}$, ${\sf C}_\mathbb{N}$, and ${\sf LPO}'$, respectively.
Here, we say that a vein $\mathbb{V}$ {\em corresponds to} a function $F$ (or to a collection $\mathcal{F}$ of functions) if for any computable flow $\Lambda$ on $\mathbb{V}$, ${\rm TP}_\Lambda$ is Weihrauch reducible to $F$ (for some $F\in\mathcal{F}$), and if (for any $F\in\mathcal{F}$) there is a computable flow $\Lambda$ on $\mathbb{V}$ such that $F$ is Weihrauch reducible to ${\rm TP}_\Lambda$.

As in the proof of Propositions \ref{prop:chara-label} and \ref{prop:weihrauch}, one can check that $\mathbb{V}_{2,1}\fr\mathbb{V}_{1,\om}$ corresponds to $({\sf C}_\mathbb{N}\star({\sf LPO}')^\ell)_{\ell\in\om}$, and that $\mathbb{V}_{0,\om}\fr\mathbb{V}_{2,1}$ corresponds to the finite parallelization $({\sf LPO}')^\ast$.
Therefore $\mathbb{V}_{1,\om}\fr\mathbb{V}_{2,1}$ (which is equivalent to $\mathbb{V}_{1,\om}\fr\mathbb{V}_{0,\om}\fr\mathbb{V}_{2,1}$) corresponds to $({\sf LPO}')^\ast\star{\sf C}_\mathbb{N}$.
Generally, it is straightforward to show the following:
\begin{enumerate}
\item $\mathbb{V}_{1,\om}\fr\mathbb{V}_{2,1}\fr\dots\fr\mathbb{V}_{1,\om}\fr\mathbb{V}_{2,1}$ corresponds to $({\sf LPO}')^\ast\star{\sf C}_\mathbb{N}\star\dots\star({\sf LPO}')^\ast\star{\sf C}_\mathbb{N}$.
\item $\mathbb{V}_{2,1}\fr\mathbb{V}_{1,\om}\fr\dots\fr\mathbb{V}_{2,1}\fr\mathbb{V}_{1,\om}$ corresponds to ${\sf C}_\mathbb{N}\star({\sf LPO}')^\ast\star\dots\star{\sf C}_\mathbb{N}\star({\sf LPO}')^\ell$ with $\ell\in\om$.
\item $\mathbb{V}_{1,1}\fr\mathbb{V}_{0,\om}\fr\mathbb{V}_{2,1}\fr\mathbb{V}_{1,\om}\fr\dots\fr\mathbb{V}_{2,1}\fr\mathbb{V}_{1,\om}$ corresponds to ${\sf C}_\mathbb{N}\star({\sf LPO}')^\ast\star\dots\star{\sf C}_\mathbb{N}\star({\sf LPO}')^\ast\star{\sf LPO}^\ell$ with $\ell\in\om$.
\end{enumerate}

\noindent Let $\mathcal{S}$ be {\sf WKL}.
It is easy to see that there is a $\Pi^0_1$ set $\mathcal{U}\subseteq 2^\om\times 2^\om$ which is Weihrauch reducible to ${\sf WWKL}$ such that $\mathcal{U}(x)$ is nonempty and has no $x$-computable element for any $x\in 2^\om$ (e.g., consider the set mentioned in Example \ref{example:non-uniformizable} (1) or $(1/2)$-${\sf WWKL}$ in \cite{BGR15}).
We then apply Theorem \ref{thm:main3} with veins mentioned in Example \ref{example:double-prime} (4) to get $F_n$ and $G_n$.
This verifies the third inequality in the statement ($\dagger$).
The first and second inequalities in ($\dagger$) require $\ell=1$ in the above item (2), which is guaranteed by our proof of Theorem \ref{thm:main3} (see Lemma \ref{lem:global-requirement}).
For the fourth inequality in ($\dagger$), we need to replace $\ell$ in the above item (3) with finite parallelization $\ast$.
This is ensured by effective compactness of $\mathcal{U}$ and $G_n$.

\section*{Acknowledgments.}
The author would like to thank Kazuto Yoshimura for valuable discussions.
The author also would like to thank the anonymous referees for their careful reading of the article and for their comments.

\bibliographystyle{plain}
\bibliography{uniformization}

\end{document}